\documentclass{article}
\usepackage[english]{babel}
\usepackage[latin1]{inputenc}
\usepackage{graphicx}
\usepackage{amssymb}
\usepackage{amsmath}
\usepackage{amsthm}
\usepackage{enumerate}
\usepackage{color}
\newtheorem{theorem}{Theorem}[section]
\newtheorem{lemma}[theorem]{Lemma}
\newtheorem{corollary}[theorem]{Corollary}

\theoremstyle{definition}
\newtheorem{remark}[theorem]{Remark}
\newtheorem{example}[theorem]{Example}
\newtheorem{examples}[theorem]{Examples}
\newtheorem{definition}[theorem]{Definition}
\numberwithin{equation}{section}

\allowdisplaybreaks[1]

\numberwithin{equation}{section}

\newcommand{\nd}{{\ensuremath d}} 

\newcommand{\uD}{\mathrm{D}}

\newcommand{\dint}{\mathrm{d}}
\newcommand{\bit}{\begin{itemize}}
\newcommand{\eit}{\end{itemize}}
\newcommand{\beq}{\begin{equation}}
\newcommand{\eeq}{\end{equation}}
\newcommand{\supp}{\mathrm{supp}\,}

\def\ls{\lesssim}

\newcommand{\nat}{\ensuremath{\mathbb{N}}}
\newcommand{\no}{\ensuremath{\nat_0}}
\newcommand{\rr}{\ensuremath{{\mathbb R}}}
\newcommand{\rd}{\ensuremath{{\mathbb R}^\nd}}
\newcommand{\cc}{\ensuremath{{\mathbb C}}}
\newcommand{\zz}{\ensuremath{{\mathbb Z}}}
\newcommand{\zd}{\ensuremath{\mathbb{Z}^\nd}}

\newcommand{\M}{{\mathcal M}_{\varphi,p}}
\newcommand{\Me}{{\mathcal M}_{u_1,p_1}}
\newcommand{\Mz}{{\mathcal M}_{u_2,p_2}}

\newcommand{\MB}{{\mathcal N}^{s}_{\varphi,p,q}}
\newcommand{\MBa}{{\mathcal N}^{s_1}_{\varphi_1,p_1,q_1}}
\newcommand{\MBb}{{\mathcal N}^{s_2}_{\varphi_2,p_2,q_2}}
\newcommand{\MBfa}{{\mathcal N}^{s_1}_{\varphi,p_1,q_1}}
\newcommand{\MBfb}{{\mathcal N}^{s_2}_{\varphi,p_2,q_2}}

\newcommand{\Gp}{{\mathcal G}_p}
\newcommand{\n}{{n}^{s}_{\varphi,p,q}}
\newcommand{\na}{{n}^{s_1}_{\varphi_1,p_1,q_1}}
\newcommand{\nb}{{n}^{s_2}_{\varphi_2,p_2,q_2}}
\newcommand{\nfa}{{n}^{s_1}_{\varphi,p_1,q_1}}
\newcommand{\nfb}{{n}^{s_2}_{\varphi,p_2,q_2}}
\newcommand{\tn}{\tilde{n}^{s}_{\varphi,p,q}}
\newcommand{\tna}{\tilde{n}^{s_1}_{\varphi_1,p_1,q_1}}
\newcommand{\tnb}{\tilde{n}^{s_2}_{\varphi_2,p_2,q_2}}

\newcommand{\whole}[1]{\ensuremath \lfloor #1 \rfloor}
\newcommand{\up}[1]{\ensuremath  \lceil #1 \rceil}

\begin{document}

\title{Wavelet decomposition and embeddings of generalised Besov-Morrey  spaces}
\date{\today}
\author{Dorothee D. Haroske\footnotemark[1],
Susana D. Moura\footnotemark[2],   and Leszek Skrzypczak\footnotemark[1] \footnotemark[3]}
\maketitle

\footnotetext[1]{The first and third author were partially supported by the German Research Foundation (DFG), Grant no. Ha 2794/8-1.}
 \footnotetext[2]{The second author was partially supported by the Centre for Mathematics of the University of Coimbra - UIDB/00324/2020, funded by the Portuguese Government through FCT/MCTES.}
\footnotetext[3]{The third author was partially supported by National Science Center, Poland,  Grant No.
  2013/10/A/ST1/00091.}

\begin{abstract} 
We study embeddings between  generalised Besov-Morrey spaces  $\MB(\rd)$. Both   sufficient and  necessary conditions for the  embeddings are proved. Embeddings of the Besov-Morrey spaces into the Lebesgue spaces $L_r(\rd)$ are also considered. Our approach requires a  wavelet characterisation of the spaces which we establish for the system of  Daubechies wavelets.
\end{abstract}

\medskip

 \smallskip


\medskip

\section{Introduction}
In this paper we study smoothness function spaces built upon generalised Morrey spaces $\M(\rd)$, $0<p<\infty$, $\varphi:(0,\infty)\rightarrow [0,\infty)$. The generalised version of Morrey spaces $\mathcal{M}_{u,p}(\rd)$, $0 < p \le u < \infty$, {was} introduced by  T.~Mizuhara \cite{mi} and E.~Nakai \cite{Nak94} in the beginning of the 1990's. The spaces were applied successfully   to PDEs, e.g.  to nondivergence elliptic differential problems, cf. \cite{FHS},  \cite{KMR} or \cite{WNTZ}, to parabolic differential equations \cite{ZJSZ} or Schr\"odinger equations \cite{KNS}.  We refer to \cite{Saw18} for further information about the spaces and the historical remarks. 
 
Also smoothness function spaces built upon Morrey spaces $\mathcal{M}_{u,p}(\rd)$,  in particular Besov-Morrey spaces  $\mathcal{N}^s_{u,p,q}(\rd)$, $0 < p \le u < \infty$, $0 < q \le \infty$, $s \in \rr$, were investigated intensively  in recent years. Yu.V.~Netrusov  was the first who combined the Besov and Morrey norms  cf. \cite{Net}. He considered function spaces on domains and proved some embedding theorem, but 
the further attention paid to the spaces was motivated first of all by possible applications to PDEs.   
The Besov-Morrey spaces $\mathcal{N}^s_{u,p,q}(\rd)$ were introduced by H.~Kozono and M.~Yamazaki in \cite{KY} and used by them to study Navier-Stokes equations. Further applications of the spaces to PDEs can be found e.g. in the  papers written by A.L.~Mazzucato \cite{Maz03}, by  L.C.F.~Ferreira, M.~Postigo \cite{FP} or by M.~Yang, Z.~Fu, J.~Sun, \cite{YFS}. 

Here we study the Besov spaces  $\MB(\rd)$ built upon generalised Morrey spaces. The spaces were introduced and studied by S.~Nakamura, T.~Noi and Y.~Sawano  \cite{NNS16}, cf. also \cite{AGNS}. In particular they proved the  atomic decomposition theorem for the spaces. In the recent paper \cite{IN19} M.~Izuki and T.~Noi investigated the spaces on domains. The generalised Besov-Morrey spaces cover Besov-Morrey spaces and local Besov-Morrey spaces considered by H.~Triebel \cite{Tri13} as special cases. 
Our main aim here is to find the sufficient and necessary conditions for the embeddings
\[ \MBa(\rd)\hookrightarrow \MBb(\rd) .\] 
Our main tools are the atomic decomposition and the wavelet characterisation. This approach allows us to consider first embeddings on the level of sequence spaces, cf. Theorem \ref{main}, and afterwards to transfer the result to function spaces, cf. Theorem \ref{fs}. In particular we regain the characterisation of embeddings of Besov-Morrey spaces $\mathcal{N}^s_{u,p,q}(\rd)$ proved  in  \cite{hs12}.  

The paper is organised as follows. In  Section~\ref{prelim} we present some preliminaries. We  recall definitions and facts needed later on.
In Section~\ref{wavelet} we obtain  the wavelet characterisation of the generalised Besov-Morrey spaces, cf. Theorem \ref{waveletth}.  Section~\ref{emb-seq} deals with the  sequence spaces $\n$ that correspond to $\MB(\rd)$ via the wavelet characterisation theorem. Theorem \ref{main}  contains the sufficient and necessary conditions for the embeddings. In the concluding Section~\ref{emb-func} we transfer the results to the function spaces. We  discuss several concrete examples.

\section{Preliminaries}\label{prelim}

First we fix some notation. By $\nat$ we denote the \emph{set of natural numbers},
by $\no$ the set $\nat \cup \{0\}$,  and by $\zd$ the \emph{set of all lattice points
in $\rd$ having integer components}. Let $\no^\nd$, where $\nd\in\nat$, be the set of all multi-indices, $\alpha = (\alpha_1, \ldots,\alpha_\nd)$ with $\alpha_j\in\no$ and $|\alpha| := \sum_{j=1}^\nd \alpha_j$. If $x=(x_1,\ldots,x_\nd)\in\rd$ and $\alpha = (\alpha_1, \ldots,\alpha_\nd)\in\no^\nd$, then we put $x^\alpha := x_1^{\alpha_1} \cdots x_\nd^{\alpha_\nd}$. 
For $a\in\rr$, let   $\whole{a}:=\max\{k\in\zz: k\leq a\}$, $\up{a} = \min\{k\in \zz:\; k\ge a \}$, and $a_{+}:=\max(a,0)$. 
Given any $u\in (0,\infty]$, it will be denoted by $u'$ the number, possible $\infty$, defined by the expression  $\frac{1}{u'}=(1-\frac{1}{u})_+$; in particular when $1\leq u\leq \infty$, $u'$ is the same as the conjugate exponent defined through $\frac{1}{u}+ \frac{1}{u'}=1$.
All unimportant positive constants will be denoted by $C$, 
occasionally the same letter $C$ is used to denote different constants  in the same chain of inequalities.
 By the notation $A \ls B$, we mean that there exists a positive constant $c$ such that
 $A \le c \,B$, whereas  the symbol $A \sim B$ stands for $A \ls B \ls A$.
 We denote by
 $|\cdot|$ the Lebesgue measure when applied to measurable subsets of $\rd$.
For each cube $Q\subset \rd $ we denote  its side length by $ \ell(Q)$, and, for $a\in (0,\infty)$, we denote by $aQ$ the cube concentric with $Q$ having the side length $a\ell(Q)$. For $x\in\rd$ and $r \in (0, \infty)$  we denote by $Q(x, r)$ the compact cube centred at $x$ with side length $r$, whose sides are parallel to the axes of coordinates. We write simply $Q(r)=Q(0,r)$ when $x=0$.
By $\mathcal{Q}$ we denote the collection of all \emph{dyadic cubes} in $\rd$, namely, $\mathcal{Q}:= \{Q_{j,k}:= 2^{-j}([0,1)^\nd+k):\ j\in\zz,\ k\in\zd\}$. 
Given two (quasi-)Banach spaces $X$ and $Y$, we write $X\hookrightarrow Y$
if $X\subset Y$ and the natural embedding of $X$ into $Y$ is continuous.

Recall first  that the classical  Morrey space
${\mathcal M}_{u,p}(\rd)$, $0<p\leq u<\infty $, is defined to be the set of all
  locally $p$-integrable functions $f\in L_p^{\mathrm{loc}}(\rd)$  such that
$$
\|f \mid {\mathcal M}_{u,p}(\rd)\| :=\, \sup_{Q\in \mathcal{Q}} |Q|^{\frac{1}{u}-\frac{1}{p}}
\left(\int_{Q} |f(y)|^p \dint y \right)^{\frac{1}{p}}\, <\, \infty\, .
$$

In this paper we consider generalised Morrey spaces where the parameter $u$ is replaced by a function $\varphi$ according to the following definition. 

\begin{definition}
Let $0<p<\infty$ and $\varphi:(0,\infty)\rightarrow [0,\infty)$ be a function which does not satisfy $\varphi\equiv 0$. Then $\M(\rd)$ is the set of all 
locally $p$-integrable functions $f\in L_p^{\mathrm{loc}}(\rd)$ for which 
\begin{equation} \label{Morrey-norm}
\|f\mid \M(\rd)\|:=\sup_{Q\in\mathcal{Q}}\varphi \bigl(\ell(Q)\bigr)\biggl(\frac{1}{|Q|}\int_Q |f(y)|^p \dint y\biggr)^{\frac{1}{p}} \, <\, \infty\, .
\end{equation}
\end{definition}

\begin{remark} \label{rmk1}
The above definition goes back to \cite{Nak94}.
When $\varphi(t)=t^{\frac{\nd}{u}}$ for $t>0$ and $0<p\leq u<\infty$, then $\M(\rd)$ coincides with ${\mathcal M}_{u,p}(\rd)$, which in turn recovers the Lebesgue space $L_p(\rd)$ when $u=p$. 
In the definition of $\|\cdot\mid \M(\rd)\|$ balls or all  cubes with sides parallel to the axes of coordinates can be taken. This  change leads to  equivalent quasi-norms. Note that for $\varphi_0\equiv 1$ (which would correspond to $u=\infty$) we obtain 
\begin{equation}\label{M1p}
{\mathcal M}_{\varphi_0,p}(\rd) = L_\infty(\rd),\quad 0<p<\infty,\quad \varphi_0\equiv 1,
\end{equation}
due to Lebesgue's differentiation theorem.  

When $\varphi(t)=t^{-\sigma} \chi_{(0,1)}(t)$ where  $-\frac{\nd}{p}\leq \sigma <0$, then $\M(\rd)$ coincides with the local Morrey spaces $\mathcal{L}^{\sigma}_p(\rd)$ introduced by H.~Triebel in \cite{Tri11}, cf. also \cite[Section~1.3.4]{Tri13}. If $ \sigma=-\frac{\nd}{p}$, then the space is a uniform Lebesgue space $\mathcal{L}_p(\rd)$.
\end{remark}

\bigskip 

For $\M(\rd)$ it is usually required that $\varphi\in \Gp$, where $ \Gp$ 
is the set of all nondecreasing functions $\varphi:(0,\infty)\rightarrow [0,\infty)$ such that $\varphi(t)t^{-\nd/p}$ is a nonincreasing function, i.e.,
$$
1\leq \frac{\varphi(r)}{\varphi(t)}\leq \left(\frac{r}{t}\right)^{\nd/p}, \quad 0<t\leq r<\infty.
$$
A justification for the use of the class $\Gp$ comes from the  lemma below, cf. e.g. \cite[Lemma 2.2]{NNS16}.  One can easily check that $\mathcal{G}_{p_2}\subset \mathcal{G}_{p_1}$ if $0<p_1\le p_2<\infty$.

\begin{lemma}[\cite{NNS16,Saw18}]
Let $0<p<\infty$ and $\varphi:(0,\infty)\rightarrow [0,\infty)$ be a function satisfying $\varphi(t_0)\neq 0$ for some $t_0>0$. 
\begin{itemize}
\item[{\upshape\bfseries (i)}] Then $\M(\rd)\neq \{0\}$ if and only if 
$$
\sup_{t>0} \varphi(t)  \min (t^{-\frac{\nd}{p}},1) < \infty.
$$
\item[{\upshape\bfseries (ii)}] Assume \ $\displaystyle\sup_{t>0} \varphi(t)  \min (t^{-\frac{\nd}{p}},1) < \infty$. Then there exists $\varphi^*\in\Gp$ such that
$$
\M(\rd)={\mathcal M}_{\varphi^*,p}(\rd)
$$
in the sense of equivalent (quasi-)norms.
 \end{itemize}
\end{lemma}

  \begin{remark}
    In \cite[Thm.~3.3]{GHLM17} it is shown that for $1\leq p_2\leq p_1<\infty$, $\varphi_i\in \mathcal{G}_{p_i}$, $i=1,2$, then
    \[ {\mathcal M}_{\varphi_1, p_1}(\rd) \hookrightarrow {\mathcal M}_{\varphi_2, p_2}(\rd)\]
    if and only if there exists some $C>0$ such that for all $t>0$, $\varphi_1(t)\leq C \varphi_2(t)$. The argument can be immediately extended to $0<p_2\leq p_1<\infty$.
    
In case of $\varphi_i(t)=t^{d/u_i}$, $0<p_i\leq u_i<\infty$, $i=1,2$, it is well-known that
\[
\Me(\rd) \hookrightarrow \Mz(\rd)\qquad\text{if and only if}\qquad
p_2\leq p_1\leq u_1=u_2, 
\]
cf. \cite{piccinini-1} and \cite{rosenthal}.
\end{remark}

We consider the following examples.

\begin{examples}\label{exm-Gp}
\begin{enumerate}[\bfseries (i)]
\item The function 
\begin{equation}\label{example1} \varphi_{u,v}(t)=
\begin{cases}
t^{\nd/u} & \text{if}\qquad t\le 1,\\ 
t^{\nd/v} & \text{if}\qquad t >1, 
\end{cases}
\end{equation}
with $0<u,v<\infty$ belongs to $\Gp$ with $p=\min(u,v)$. In particular, taking $u=v$, the
function $\varphi(t)=t^{\frac{\nd}{u}}$ belongs to $\mathcal{G}_p$ whenever  $0<p\leq u<\infty$. 
\item The function  $\varphi(t)=\max(t^{\nd/v},1)$ belongs to $\mathcal{G}_v$. It corresponds to \eqref{example1} with $u=\infty$.
\item The function  $\varphi(t)=\sup\{s^{\nd/u}\chi_{(0,1)}(s): s\le t\}= \min(t^{\nd/u},1)$ defines an equivalent (quasi)-norm in $\mathcal{L}^{\sigma}_p(\rd)$, $\sigma= -\frac{\nd}{u}$, $p\le u$. The function $\varphi$  belongs to $\mathcal{G}_u\subset \mathcal{G}_p$. It corresponds to \eqref{example1} with $v=\infty$. 
\item  The   function $\varphi(t)=t^{\nd/u}(\log (L+t))^a$, with $L$ being a sufficiently large constant, belongs to $\mathcal{G}_u$ if  $0< u<\infty$ and $a\leq 0$. 
\end{enumerate}
\smallskip
Other examples   can be found e.g. in \cite[Ex.~3.15]{Saw18}.
\end{examples}

Let $\mathcal{S}(\rd)$ be the set of all Schwartz functions on $\rd$, endowed
with the usual topology,
and denote by $\mathcal{S}'(\rd)$ its topological dual, namely,
the space of all bounded linear functionals on $\mathcal{S}(\rd)$
endowed with the weak $\ast$-topology.
For all $f\in \mathcal{S}(\rd)$ or $f\in\mathcal{S}'(\rd)$, we
use $\mathcal{F} f $ to denote its Fourier transform, and $\mathcal{F}^{-1}f$ for its inverse.
Now let us define the generalised Besov-Morrey spaces introduced in \cite{NNS16}.

Let $\eta_0,\eta\in \mathcal{S}(\rd)$ be nonnegative compactly supported functions satisfying
\begin{equation*}
\eta_0(x)>0 \quad \text{if}\quad x \in Q(2),
\end{equation*}
\begin{equation*}
0\notin \supp \eta\quad \text{and} \quad \eta(x)>0 \quad \text{if}\quad x \in Q(2) \setminus Q(1).
\end{equation*}
For $j\in\nat$, let $\eta_j(x):=\eta(2^{-j}x)$,  $x\in\rd$.

\begin{definition} \label{def-spaces} 
Let $0<p<\infty$, $0<q\leq \infty$, $s\in \rr$, and $\varphi\in \Gp$. 
The  generalised Besov-Morrey   space
  $\MB(\rd)$ is defined to be the set of all  $f\in\mathcal{S}'(\rd)$ such that
\begin{align*}
\big\|f\mid \MB(\rd)\big\|:=
\bigg(\sum_{j=0}^{\infty}2^{jsq}\big\| \mathcal{F}^{-1}(\eta_j  \mathcal{F} f)\mid
\M(\rd)\big\|^q \bigg)^{1/q} < \infty,
\end{align*}
with the usual modification made in case of $q=\infty$.
\end{definition}

\begin{remark}\label{rem-coinc}
 The above spaces have been introduced in \cite{NNS16}. There the authors have proved that those spaces are independent of the choice of the functions $\eta_0$ and $\eta$ considered in the definition, as different choices lead to equivalent quasi-norms, cf.~\cite[Thm~1.4]{NNS16}. 
 
  When $\varphi(t)=t^{\frac{\nd}{u}}$ for $t>0$ and $0<p\leq u<\infty$, then 
$$
\MB(\rd)={\mathcal N}^s_{u,p,q}(\rd) 
$$
are the usual Besov-Morrey, 
which are studied in \cite{YSY10} or in the recent survey papers by W.~Sickel \cite{s011,s011a}.  
Of course, we can recover the  classical Besov spaces $B^s_{p,q}(\rd)$ 
for any $0<p<\infty$, $0<q\leq \infty$, and $s\in \rr$, since 
$$
B^s_{p,q}(\rd)={\mathcal N}^s_{p,p,q}(\rd) . 
$$
 When  $\varphi(t)=\min(t^\frac{\nd}{u},1)$, then we recover  the local Besov-Morrey spaces introduced by H.~Triebel, 
\[ \MB(\rd) = B^s_q(\mathcal{L}^\sigma_p, \rd), \quad\sigma= -\frac{\nd}{u}, \quad p\le u, \] 
cf. \cite[Section~1.3.4]{Tri13}.

Besides the elementary embeddings
$$\mathcal  N^{s+\varepsilon}_{\varphi,p,q_1} (\rd)\hookrightarrow
\mathcal  N^{s}_{\varphi,p,q_2}(\rd), \quad \varepsilon>0,$$
and 
$$\mathcal  N^{s}_{\varphi,p,q_1} (\rd) \hookrightarrow \mathcal  N^{s}_{\varphi,p,q_2}(\rd) , \quad q_1\leq q_2,$$
cf. \cite[Prop.~3.3]{NNS16}, we can  also 
prove that 
$$
\mathcal{N}^0_{\varphi,p,\min\{p,2\}} (\rd)\hookrightarrow \, \M(\rd)\, \hookrightarrow \,{\cal N}^0_{\varphi,p,\infty} (\rd)\qquad \text{if}\qquad 1<p<\infty, 
$$
when $\varphi$ satisfies the additional condition
$$
c\left(\frac{r}{t}\right)^{\varepsilon}\leq \frac{\varphi(r)}{\varphi(t)}, \quad 0<t\leq r<\infty,
$$
for some constants $\varepsilon>0$ and $c>0$.
This is a consequence of  Corollary 6.17 of \cite{NNS16}.
\end{remark}

\subsubsection*{The  atomic decomposition}

An important tool in our later considerations is the characterisation of the generalised Besov-Morrey spaces by means of atomic decompositions.
We follow \cite{NNS16} and start by defining the appropriate sequence spaces and atoms.

\begin{definition}\label{def-seq-spaces} 
Let $0<p<\infty$, $0<q\leq \infty$, $s\in \rr$, and $\varphi\in \Gp$. 
The  generalised Besov-Morrey sequence space $\n(\rd)$ is the set of all double-indexed sequences 
$\lambda:=\{\lambda_{j,m}\}_{j\in\no,m\in\zd}\subset \cc$ for which the quasi-norm 
\begin{equation} \label{seq-besov}
\| \lambda\mid \n\|:= \biggl(\sum_{j=0}^{\infty}2^{jsq}   \Big\| \sum_{m\in\zd} \lambda_{j,m} \chi_{Q_{j,m}} \,\Big|\, \M(\rd) \Big\|^q \biggr)^{1/q} 
\end{equation}
is finite (with the usual modification if $q=\infty$).
\end{definition}
\smallskip

\begin{remark} \label{seq-besov-rem} When $\varphi(t)=t^{\frac{\nd}{u}}$ for $t>0$ and $0<p\leq u<\infty$, then 
$$ 
\n(\rd)={n}^s_{u,p,q}(\rd) 
$$
are the usual Besov-Morrey sequence spaces. 
Moreover if $u=p$, then the space $\n(\rd)$ coincides with a classical Besov sequence space  $b^s_{p,q}(\rd)$ since $\M(\rd)=L_p(\rd)$ in that case. 
\end{remark} 

\smallskip

\begin{definition}\label{def-atoms} 
Let $L\in\no\cup\{-1\}$, $K\in\no$, and $c>1$. 
A $C^K$-function $a:\rd\rightarrow \cc$ is said to be a $(K,L,c)$-atom centered at  $Q_{j,m}$, where $j\in\no$ and $m\in\zd$, if 
\begin{equation}
2^{-j|\alpha|} |\uD^{\alpha}a(x)| \leq \chi_{cQ_{j,m}}(x)
\end{equation}
 for all  $x\in\rd$  and for all  $\alpha\in \no^\nd$ with $|\alpha|\leq K$, and when 
\begin{equation}  
   \int_{\rd} x^{\beta} a(x) \dint x=0,
\end{equation}
for all  $\beta\in \no^\nd$ with $|\beta|\leq L$ when $L\geq 0$.
In the sequel we write $a_{j,m}$ instead of $a$ if the atom is located at  $Q_{j,m}$, i.e.,  $\supp a_{j,m}\subset cQ_{j,m}$.
\end{definition}

We use the notation
\[
\sigma_p:=\nd\left(\frac{1}{\min (1,p)}-1\right),\quad 0<p\leq\infty,
\]
in the sequel. 
The following result coincides with \cite[Thm.~4.4]{NNS16}, cf. also \cite[Rmk.~4.3]{NNS16}, see also \cite[Thm.~10.15]{LSUYY13}.

\begin{theorem}  \label{atomic-decomposition}
Let $0<p<\infty$, $0<q\leq \infty$, $s\in \rr$, and $\varphi\in \Gp$. 
Let also $c>1$, $L\in\no\cup\{-1\}$ and $K\in\no$ be such that
$$
K\geq \whole{1+s}_+ \quad  \text{and} \quad L\geq  \max(-1,\whole{\sigma_p-s}) .
$$
\begin{itemize}
\item[{\upshape\bfseries (i)}] Let $f\in\MB(\rd)$. Then there exists a family $\{a_{j,m}\}_{j\in\no, m\in\zd}$ of $(K,L,c)$-atoms and a sequence $\lambda=\{\lambda_{j,m}\}_{j\in\no,m\in\zd}\in\n(\rd)$ such that
$$
f=\sum_{j=0}^{\infty}\sum_{m\in\zd} \lambda_{j,m} a_{j,m} \qquad \text{in }\quad {\mathcal S}'(\rd)
$$
and 
$$
\|\lambda\mid \n(\rd)\|\lesssim \|f\mid \MB(\rd)\|.
$$
\item[{\upshape\bfseries (ii)}] Let $\{a_{j,m}\}_{j\in\no, m\in\zd}$ be a family of $(K,L,c)$-atoms and $\lambda=\{\lambda_{j,m}\}_{j\in\no,m\in\zd}\in\n(\rd)$. Then 
$$
f=\sum_{j=0}^{\infty}\sum_{m\in\zd} \lambda_{j,m} a_{j,m} 
$$
converges in ${\mathcal S}'(\rd)$ and belongs to $\MB(\rd)$. Furthermore
$$
\|f\mid \MB(\rd)\| \lesssim \|\lambda\mid \n(\rd)\|.
$$
\end{itemize}
\end{theorem}

{The next lemma will be useful in the sequel and shows that the sequence spaces $\n$ can be defined through a more convenient  equivalent norm, extending the result for  $n^s_{u,p,q}$ from \cite[Prop.~3.1]{hs12}.

\begin{lemma} \label{equiv-norm}
Let $0<p<\infty$, $0<q\leq \infty$, $s\in \rr$, and $\varphi\in \Gp$. Then 
$$
\n(\rd)= \{\lambda=\{\lambda_{j,m}\}_{j,m}:  \| \lambda\mid \n\|^*<\infty \}
$$
where
$$
\| \lambda\mid \n\|^*:=\Bigg(\sum_{j=0}^\infty 2^{jsq} 
\mathop{\sup_{\nu: \nu \leq j}}_{k\in \zd}\!  \varphi(2^{-\nu})^q \,2^{(\nu-j)\frac{\nd}{p} q}\Big(\!\! \mathop{\sum_{m\in\zd:}}_{Q_{j,m}\subset Q_{\nu,k}}\!\!|\lambda_{j,m}|^p\Big)^{\frac q p}\Bigg)^{1/q}
$$
with the usual modification if $q=\infty$.
\end{lemma}

\begin{proof}
For each $j\in\no$  we calculate the quasi-norm
$$
 \Big\| \sum_{m\in\zd} \lambda_{j,m} \chi_{Q_{j,m}} \,\Big|\, \M(\rd) \Big\|.
$$
Let $Q=Q_{\nu, k}$, $\nu\in\zz$, $k\in\zd$, be a dyadic cube. If $j\geq \nu$, then
\begin{equation} \label{seq:1}
 \varphi \bigl(\ell(Q)\bigr)\biggl(\frac{1}{|Q|} \int_Q  \Bigl| \sum_{m\in\zd} \lambda_{j,m} \chi_{Q_{j,m}}(x)\Big|^p \ \dint x\biggr)^{1/p} 
 = \varphi (2^{-\nu})  2^{(\nu-j)\frac{\nd}{p}}\biggl(\mathop{\sum_{m\in\zd:}}_{Q_{j,m}\subset Q_{\nu, k}}  
 | \lambda_{j,m}|^p \biggr)^{1/p} .
\end{equation}
If $j<\nu$, there exists only one $m_0\in\zd$ such that $Q=Q_{\nu, k}\subset Q_{j, m_0}$, and, moreover, since $\varphi$ is nondecreasing, we obtain
\begin{equation} \label{seq:2}
 \varphi \bigl(\ell(Q)\bigr)\biggl(\frac{1}{|Q|} \int_Q  \Bigl| \sum_{m\in\zd} \lambda_{j,m} \chi_{Q_{j,m}}(x)\Big|^p \dint x\biggr)^{1/p} 
=  \varphi (2^{-\nu})  
\biggl(\frac{1}{|Q|} \int_Q |\lambda_{j,m_0}|^p \dint x\biggr)^{1/p} 
 \leq   \varphi (2^{-j})  |\lambda_{j,m_0}|.
\end{equation}
From \eqref{seq:1} and  \eqref{seq:2} we immediately have 
$$
 \Big\| \sum_{m\in\zd} \lambda_{j,m} \chi_{Q_{j,m}} \,\Big|\, \M(\rd) \Big\| 
 \leq   \mathop{\sup_{\nu: \nu \leq j}}_{k\in \zd} \varphi (2^{-\nu})  2^{(\nu-j)\frac{\nd}{p}}\biggl(\mathop{\sum_{m\in\zd:}}_{Q_{j,m}\subset Q_{\nu, k}}  
 | \lambda_{j,m}|^p \biggr)^{1/p}. 
$$
The reverse inequality is clear, from the definition of $\|\cdot\mid \M(\rd)\|$  and  \eqref{seq:1}. Therefore 
\begin{equation} \label{seq:3}
 \Big\| \sum_{m\in\zd} \lambda_{j,m} \chi_{Q_{j,m}} \,\Big|\, \M(\rd) \Big\| = \mathop{\sup_{\nu: \nu \leq j}}_{k\in \zd} \varphi (2^{-\nu})  2^{(\nu-j)\frac{\nd}{p}}\biggl(\mathop{\sum_{m\in\zd:}}_{Q_{j,m}\subset Q_{\nu, k}}   | \lambda_{j,m}|^p \biggr)^{1/p}.
\end{equation}
The result follows from \eqref{seq:3} taking into account \eqref{seq-besov}.
\end{proof}
}

\section{The  wavelet characterisation}\label{wavelet}
We assume that the reader is familiar with the basic notation and  assertions of the wavelet theory. There is a variety of excellent books that present  general background material on wavelets, we can refer, in particular, to \cite{Dau2}, \cite{HW96} and \cite{Woj}. We will follow the approach presented in \cite{HST} and consider here the compactly supported Daubechies wavelets.

Let $L\in \mathbb{N}$ and let   $\psi_F,\psi_M\in C^L(\mathbb{R})$ are real-valued compactly supported ($L^2$-normalised) functions with
\begin{equation}
\int_\rr \psi_F^2(t) \dint t =1 ,\qquad \int_\rr \psi_M(t)t^\ell \dint t = 0, \qquad \ell< L. 
\end{equation} 
The function  $\psi_F$ is called   scaling function (or father wavelet) and  $\psi_M $ is called  an associated  function (mother wavelet).

Let $G=(G_1,...,G_\nd)\in G^*=\left\{F,M\right\}^{\nd *}$, where $^*$ indicates that at least one of the components of $G$ must be an $M$. 
Then we set 
\begin{equation}\label{eq2:1}
\psi_{j,m}^G = 2^{j\nd/2} \prod_{r=1}^\nd\psi^{G_r}(2^jx_r-m_r), \qquad \psi_m(x)= \prod_{r=1}^\nd\psi_{F}(x_r-m_r),
\end{equation}
where $j\in\mathbb{N}_0$, $m\in\mathbb{Z}^\nd$, $G\in G^*$.  The family $\{\psi_m,\psi_{j,m}^G: j\in\mathbb{N}_0, \ m\in\zd,\ G\in G^*\}$  is called a (Daubechies) wavelet system.

We will need the following modified version $\tn(\rd)$  of $\n(\rd)$ spaces. The space $\tn(\rd) $ collects all sequences 
\[
\lambda= \left\{ \lambda_m\in \mathbb{C}, \lambda_{j,m}^G \in \mathbb{C} : \; m\in \zd, \; j \in \mathbb{N}_0,\; G\in G^* \right\}
\]
quasi-normed by 
\begin{equation}\label{norm-tn}
\| \lambda\mid \tn(\rd)\|:= \Big\| \sum_{m\in\zd} \lambda_{m} \chi_{Q_{m}} \,\Big|\, \M(\rd) \Big\|  + \sum_{G\in G^*} \Big\|
\{\lambda_{j,m}^G\}| \n(\rd)\Big\| . 
\end{equation}

\begin{theorem}\label{waveletth}
Let $0<p<\infty$, $0<q\leq \infty$, $s\in \rr$, and $\varphi\in \Gp$.  For the wavelets defined in \eqref{eq2:1} we take
\begin{equation}\label{L-cond-wave}
L> \max\{ \lfloor 1+s\rfloor_+, \frac{\nd}{p} -s\}.
\end{equation}
Let $f\in \mathcal{S}'(\rd)$. Then $f\in \MB(\rd)$ if and only if it can be represented as
\begin{equation*}
f=  \sum_{m\in \zd}\lambda_{m}\psi_{m} \;+\;\sum_{G\in G^*}\sum_{j\in \mathbb{N}_0} \sum_{m\in \zd}\lambda_{jm}^G2^{-j\nd/2}\psi_{j,m}^G,\qquad \lambda\in \tn(\rd),
\end{equation*}
unconditional convergence being in $\mathcal{S}'(\rd)$. 
The representation is unique with
\begin{equation*}
\lambda_{j,m}^G = \lambda_{j,m}^G(f)=  2^{j\nd/2}\left(f,\psi_{j,m}^G\right)\quad\text{and}\quad \lambda_{m} = \lambda_{m}(f)=  \left(f,\psi_{m}\right),
\end{equation*}
and
\begin{equation*}
I:f\mapsto \{\lambda_{m}(f), 2^{j\nd/2}(f,\psi^G_{j,m})\}
\end{equation*}
is a linear isomorphism of $\MB(\rd)$ onto $\tn(\rd)$.
\end{theorem}

\begin{proof}
\emph{Step 1.}
We prove that the theorem follows from Theorem 5.1 in \cite{HST}. The space  $\MB(\rd)$ is an (isotropic, inhomogeneous) quasi-Banach function space which satisfies
\[ \mathcal{S}(\rd)\hookrightarrow \MB(\rd)\hookrightarrow \mathcal{S}'(\rd) \] 
and  which can be characterised in terms of an $L$-atomic decomposition with $L=K$, cf. Theorem \ref{atomic-decomposition}. Please note that the inequality $L>\frac{\nd}{p}-s$ implies $L>\sigma_p -s$.  So it is sufficient to prove that the sequence space $\tn(\rd)$ is a $\varkappa$-sequence space for some $\varkappa$, $0<\varkappa<L$, cf. Definition 4.1 in \cite {HST}. One can easily check that it is sufficient to prove that the space $\n(\rd)$ is a $\varkappa$-space, so to simplify the notation we restrict our attention to the  space $\n=\n(\rd)$.

Let $b>1$ and $C_1>0$.  For $j,J\in \no$ and $m,M\in \zd$  we put 
\[ I^j_J(m) = \left\{ M\in \zd: \; b Q_{J,M}\cap C_1Q_{j,m}\not= \emptyset\right\} \quad\text{and}\quad 
 \hat{I}^j_J(M) = \left\{ m\in \zd: \; b Q_{J,M}\cap C_1Q_{j,m}\not= \emptyset\right\} .\]

Note that the cardinalities of $\#I^j_J(m)$ and $ \#\hat{I}^j_J({M})$ satisfy 
\begin{equation}\label{kappa0}
\#I^j_J(m)\sim 
\begin{cases}
1, & \quad J\le j,\\
2^{\nd(J-j)},&\quad  J>j, 
\end{cases} 
\qquad \text{and}\qquad
\#\hat{I}^j_J(M)\sim 
\begin{cases}
1, & \quad j\le J,\\
2^{\nd(j-J)},&\quad j>J. 
\end{cases}
\end{equation}

We prove that there exists $\varkappa$, $0<\varkappa< L$, such that
\begin{itemize}
\item[{(i)}] for any $b>1$, $C_1>0$, and all $\mu\in\n 
  $, any sequence $\lambda=\{\lambda_{j,m}\}$ with 
\begin{align}\label{kappa1}
|\lambda_{j,m}| \le C_1 \sum_{J\in \mathbb{N}_0} 2^{-\varkappa |J-j|} \sum_{M\in I^j_J(m)} 2^{-\nd(J-j)_+} |\mu_{J,M}|, \qquad 
j\in \mathbb{N}_0, \quad m\in \zd,
\end{align} 
belongs to $\n$ 
and satisfies
 \begin{equation}\label{kappa2}
 \|\lambda|\n\| \le C  \|\mu|\n\| ;
 \end{equation}
\item[{(ii)}] for any cube $Q$ there is a constant $c_Q>0$ such that for all $\mu\in \n$
 \begin{align}\label{kappa3}
 |\mu_{J,M}| \le C_Q 2^{J\varkappa} \|\mu|\n\|\quad\text{for all} \quad J\in \mathbb{N}_0, \quad M\in \zd\quad\text{with}\quad Q_{J,M}\subset Q.
 \end{align}
\end{itemize}
Please note that  there is a constant  $\eta=\eta(b,C_1)\in \mathbb{N}$ such that if  $Q_{j,m}\subset Q_{\nu,k}$  and $bQ_{j,m}\cap C_1Q_{J,M}\not= \emptyset$, then $Q_{J,M}\subset Q_{\nu-\eta,\ell}$ for some  dyadic cube $Q_{\nu-\eta,\ell}$, $\ell=\ell(k)$, such that $ Q_{\nu,k}\subset Q_{\nu-\eta,\ell}$.

 \emph{Step 2.} We prove the property  (i)  for $0<p\le 1$. We decompose the sum in \eqref{kappa1} into two parts for  $J\le j$ and  for $J>j$. Let $\nu\in {\zz}$ and $k\in\zd$ be fixed, {with $\nu\leq j$}. Then 
\begin{align}\label{kappa4}
\lefteqn{{ \mathop{\sum_{m\in\zd:}}_{Q_{j,m}\subset Q_{\nu,k}}}
\!\! |\lambda_{j,m}|^p }\\ \nonumber
\le & \  C_1  \sum_{J=0}^j 2^{-\varkappa (j-J)p}  {\mathop{\sum_{m\in\zd:}}_{Q_{j,m}\subset Q_{\nu,k}}} 
 \sum_{M\in I^j_J(m)} |\mu_{J,M}|^p + 
C_1 \sum_{J=j+1}^\infty 2^{-(\varkappa+\nd) (J-j)p} 
{ \mathop{\sum_{m\in\zd:}}_{Q_{j,m}\subset Q_{\nu,k}}}
\sum_{M\in I^j_J(m)} |\mu_{J,M}|^p  \\ \nonumber 
\le & \ C_1  \sum_{J=0}^j 2^{-\varkappa (j-J)p}   \!\!\!  { \mathop{\sum_{M\in\zd:}}_{Q_{J,M}\subset Q_{\nu-\eta,\ell}}} \!\!\!\! 
\#\hat{I}^j_J(M)\, |\mu_{J,M}|^p   + 
C_1  \sum_{J=j+1}^\infty 2^{-(\varkappa+\nd) (J-j)p}  \!\!\! { \mathop{\sum_{M\in\zd:}}_{Q_{J,M}\subset Q_{\nu-\eta,\ell}} } \!\!\!\! 
\#\hat{I}^j_J(M) \, |\mu_{J,M}|^p \\ \nonumber
\le   
 & C_1   \sum_{J=0}^j 2^{-(\varkappa-\frac{\nd}{p}) (j-J)p}  \!\!\! { \mathop{\sum_{M\in\zd:}}_{Q_{J,M}\subset Q_{\nu-\eta,\ell}} } \!\!\!\!  
  |\mu_{J,M}|^p   + 
C_1 \sum_{J=j+1}^\infty 2^{-(\varkappa+\nd) (J-j)p} \!\!\! { \mathop{\sum_{M\in\zd:}}_{Q_{J,M}\subset Q_{\nu-\eta,\ell}}} \!\!\!\! 
 |\mu_{J,M}|^p  ,
\end{align}
where the last inequality follows from \eqref{kappa0} and the last but one follows from the definition of the set $\hat{I}^j_J(M)$. 
 
If $\frac{q}{p}\le 1$, then 
\begin{align}\label{kappa5}
&\Big(\!\!\!{ \mathop{\sum_{m\in\zd:}}_{Q_{j,m}\subset Q_{\nu,k}}} 
\!\!\!|\lambda_{j,m}|^p\Big)^\frac{q}{p} \le \\ \nonumber
 &\qquad \le C_1 \sum_{J=0}^j 2^{-(\varkappa-\frac{\nd}{p}) (j-J)q} \,\Big( \!\!\!\!\!{ \mathop{\sum_{M\in\zd:}}_{Q_{J,M}\subset Q_{\nu-\eta,\ell}}} \!\!\!\! 
 |\mu_{J,M}|^p\Big)^\frac{q}{p}   + 
 \leq C_1\sum_{J=j+1}^\infty 2^{-(\varkappa+\nd) (J-j)q} \,\Big( \!\!\!\!\! { \mathop{\sum_{M\in\zd:}}_{Q_{J,M}\subset Q_{\nu-\eta,\ell}}} \!\!\!\! 
  |\mu_{J,M}|^p\Big)^\frac{q}{p}.   
\end{align}
If $\frac{q}{p}> 1$, then for any $\varepsilon>0$ we get, using the H\"older inequality,
\begin{align}\label{kappa6}
&\Big(\!\!\!\!{ \mathop{\sum_{m\in\zd:}}_{Q_{j,m}\subset Q_{\nu,k}}} 
\!\!\!|\lambda_{j,m}|^p\Big)^\frac{q}{p} \le \\ \nonumber
 &\qquad \le C_1 \sum_{J=0}^j 2^{-(\varkappa-\varepsilon-\frac{\nd}{p}) (j-J)q} \Big( \!\!\!\!\!\! { \mathop{\sum_{M\in\zd:}}_{Q_{J,M}\subset Q_{\nu-\eta,\ell}} } \!\!\!\! 
  |\mu_{J,M}|^p\Big)^\frac{q}{p}   + 
C_1  \sum_{J=j+1}^\infty 2^{-(\varkappa-\varepsilon+\nd) (J-j)q}\Big( \!\!\!\!\! \!{ \mathop{\sum_{M\in\zd:}}_{Q_{J,M}\subset Q_{\nu-\eta,\ell}} } \!\!\!\! 
 |\mu_{J,M}|^p\Big)^\frac{q}{p}.   
\end{align}
So in both cases we have 
\begin{align}\label{kappa7}
& \mathop{\sup_{\nu: \nu \leq j}}_{k\in \zd}\!  \varphi(2^{-\nu})^q \,2^{\nu \nd\frac{q}{p}} \Big(\!\!\! { \mathop{\sum_{m\in\zd:}}_{Q_{j,m}\subset Q_{\nu,k}}}
|\lambda_{j,m}|^p\Big)^\frac{q}{p}  \\ \nonumber
  \le &\quad 2^{\eta \nd\frac{q}{p}} C_1 \sum_{J=0}^j 2^{-(\varkappa-\varepsilon-\frac{\nd}{p}) (j-J)q} \mathop{\sup_{\nu: \nu \leq j}}_{\ell\in \zd}\!  \varphi(2^{-\nu+\eta})^q \,2^{(\nu-\eta) \nd\frac{q}{p}} \Big(\!\!\! { \mathop{\sum_{M\in\zd:}}_{Q_{J,M}\subset Q_{\nu-\eta,\ell}}} \!\!\!\!
   |\mu_{J,M}|^p\Big)^\frac{q}{p}   + \\
  \nonumber &\quad 
  2^{\eta \nd\frac{q}{p}}  C_1 \sum_{J=j+1}^\infty 2^{-(\varkappa-\varepsilon+\nd) (J-j)q} \mathop{\sup_{\nu: \nu \leq j}}_{\ell\in \zd}\!  \varphi(2^{-\nu+\eta})^q \,2^{(\nu-\eta) \nd\frac{q}{p}} \Big( \!\!\!\!\! { \mathop{\sum_{M\in\zd:}}_{Q_{J,M}\subset Q_{\nu-\eta,\ell}}} \!\!\!\!
   |\mu_{J,M}|^p\Big)^\frac{q}{p}  \\
  \nonumber \le &\quad
  2^{\eta \nd\frac{q}{p}} C_1 \sum_{J=0}^j 2^{-(\varkappa-\varepsilon-\frac{\nd}{p}) (j-J)q} \mathop{\sup_{\nu: \nu \leq J}}_{k\in \zd}\!  \varphi(2^{-\nu})^q \,2^{\nu \nd\frac{q}{p}} \Big( \!\!\!\!\! { \mathop{\sum_{M\in\zd:}}_{Q_{J,M}\subset Q_{\nu,k}}} \!\!\!\!
   |\mu_{J,M}|^p\Big)^\frac{q}{p}   + \\ \nonumber
  &\quad 
2^{\eta \nd\frac{q}{p}}  C_1 \sum_{J=j+1}^\infty 2^{-(\varkappa-\varepsilon+\nd) (J-j)q} \mathop{\sup_{\nu: \nu \leq J}}_{k\in \zd}\!  \varphi(2^{-\nu})^q \,2^{\nu \nd\frac{q}{p}} \Big( \!\!\!\!\!  { \mathop{\sum_{M\in\zd:}}_{Q_{J,M}\subset Q_{\nu,k}}} \!\!\!\!
  |\mu_{J,M}|^p\Big)^\frac{q}{p}.   
\end{align}
The first inequality follows from {\eqref{kappa5} and } \eqref{kappa6} and the fact that any $k$ appoints one $\ell=\ell(k)$, so the supremum  over $k$ can be dominated by the supremum over $\ell$. The second inequality follows  by rescaling.

In consequence, 
 \begin{align}\label{kappa8}
 & \|\lambda|\n\|^q   \le    c \sum_{j=0}^\infty 2^{j(s-\frac{\nd}{p}) q} 
 \sum_{J=0}^j 2^{-(\varkappa-\varepsilon-\frac{\nd}{p}) (j-J)q} \mathop{\sup_{\nu: \nu \leq J}}_{k\in \zd}\!  \varphi(2^{-\nu})^q \,2^{\nu \nd\frac{q}{p}} \Big(
  \!\!\!\!\!  { \mathop{\sum_{M\in\zd:}}_{Q_{J,M}\subset Q_{\nu,k}}} \!\!\!\! 
   |\mu_{J,M}|^p\Big)^\frac{q}{p} \\ \nonumber 
& \qquad  +  c \sum_{j=0}^\infty 2^{j(s-\frac{\nd}{p}) q} 
\sum_{J=j+1}^\infty 2^{-(\varkappa-\varepsilon+\nd) (J-j)q} \mathop{\sup_{\nu: \nu \leq J}}_{k\in \zd}\!  \varphi(2^{-\nu})^q \,2^{\nu \nd\frac{q}{p}}
 \Big( \!\!\!\!\!  { \mathop{\sum_{M\in\zd:}}_{Q_{J,M}\subset Q_{\nu,k}}} \!\!\!\! 
  |\mu_{J,M}|^p\Big)^\frac{q}{p}     \\ \nonumber 
& \qquad \le c \sum_{j=0}^\infty 
 \sum_{J=0}^j  2^{J(s-\frac{\nd}{p}) q} 2^{-(\varkappa-\varepsilon-s)(j-J)q} \mathop{\sup_{\nu: \nu \leq J}}_{k\in \zd}\!  \varphi(2^{-\nu})^q \,2^{\nu \nd\frac{q}{p}} 
 \Big( \!\!\!\!\!  { \mathop{\sum_{M\in\zd:}}_{Q_{J,M}\subset Q_{\nu,k}}} \!\!\!\! 
  |\mu_{J,M}|^p\Big)^\frac{q}{p}  
\\\nonumber &\qquad + 
c\sum_{j=0}^\infty  
\sum_{J=j+1}^\infty 2^{J(s-\frac{\nd}{p}) q} 2^{-(\varkappa-\varepsilon- \sigma_p+s) (J-j)q} \mathop{\sup_{\nu: \nu \leq J}}_{k\in \zd}\!  \varphi(2^{-\nu})^q \,2^{\nu \nd\frac{q}{p}} \Big( \!\!\!\!\!  { \mathop{\sum_{M\in\zd:}}_{Q_{J,M}\subset Q_{\nu,k}}} \!\!\!\! 
 |\mu_{J,M}|^p\Big)^\frac{q}{p}   
\\
\nonumber &\qquad \le
c \sum_{J=0}^\infty  
2^{J(s-\frac{\nd}{p}) q}  \mathop{\sup_{\nu: \nu \leq J}}_{k\in \zd}\!  \varphi(2^{-\nu})^q \,2^{\nu \nd\frac{q}{p}} 
\Big( \!\!\!\!\!  { \mathop{\sum_{M\in\zd:}}_{Q_{J,M}\subset Q_{\nu,k}}} \!\!\!\! 
|\mu_{J,M}|^p\Big)^\frac{q}{p} \sum_{j= J}^\infty  2^{-(\varkappa-\varepsilon-s)(j-J)q} \\\nonumber
&\qquad +
c\sum_{J=1}^\infty 
2^{J(s-\frac{\nd}{p}) q}  \mathop{\sup_{\nu: \nu \leq J}}_{k\in \zd}\!  \varphi(2^{-\nu})^q \,2^{\nu \nd\frac{q}{p}} \Big(
 \!\!\!\!\!  { \mathop{\sum_{M\in\zd:}}_{Q_{J,M}\subset Q_{\nu,k}}} \!\!\!\!
  |\mu_{J,M}|^p\Big)^\frac{q}{p} \sum_{j=0}^{J-1} 2^{-(\varkappa-\varepsilon- \sigma_p+s) (J-j)q} 
 \\ &\qquad  \le c \|\mu|\n\|^q \nonumber 
 \end{align}
 if we choose $\varepsilon >0$ such that  $\varkappa-\varepsilon- \sigma_p+s>0$ and $ \varkappa-\varepsilon-s>0$. This is always possible if $\varkappa>\max\{\sigma_p-s,s\}$, that is, we need $L>\max\{\sigma_p-s,s\}$ here which is implied by \eqref{L-cond-wave}. This finishes the proof of (i) for $0<p\le 1$. 

\emph{Step 3}. Now we prove the property  (i)  for $p> 1$. Applying the H\"older inequality twice yields for some $\varepsilon>0$, 
\begin{align}\label{kappa9}
\sum_{J=0}^j 2^{-\varkappa(j-J)} \sum_{M\in I^j_J(m)} |\mu_{J,M}| \le c_2 \Big( \sum_{J=0}^j 2^{-(\varkappa-\varepsilon)(j-J)p} \sum_{M\in I^j_J(m)} |\mu_{J,M}|^p\Big)^\frac{1}{p} 
\intertext{and}
\sum_{J=j+1}^\infty 2^{-(\varkappa+ \nd)(J-j)} \sum_{M\in I^j_J(m)} |\mu_{J,M}| \le c_2 \Big( \sum_{J=j+1}^\infty 2^{-(\varkappa-\varepsilon+\frac{\nd}{p})(J-j)p} \sum_{M\in I^j_J(m)} |\mu_{J,M}|^p\Big)^\frac{1}{p}, \label{kappa10}
\end{align}
in view of \eqref{kappa0}, see also \cite{HST}. 
 Now using \eqref{kappa9} and \eqref{kappa10} and a similar method as in \eqref{kappa4} we can prove that 
  \begin{align}\label{kappa11}
{ \mathop{\sum_{m\in\zd:}}_{Q_{j,m}\subset Q_{\nu,k}}} 
|\lambda_{j,m}|^p \le & \      
 C_1 \sum_{J=0}^j 2^{-(\varkappa-\varepsilon-\frac{\nd}{p}) (j-J)p}  \!\!\!\! { \mathop{\sum_{M\in\zd:}}_{Q_{J,M}\subset Q_{\nu-\eta,\ell}}} \!\!\!\!  
 |\mu_{J,M}|^p   + 
\\ \nonumber &\qquad \qquad \nonumber
 C_1 \sum_{J=j+1}^\infty 2^{-(\varkappa-\varepsilon+\frac{\nd}{p}) (J-j)p}  \!\!\!\!\! { \mathop{\sum_{M\in\zd:}}_{Q_{J,M}\subset Q_{\nu-\eta,\ell}}} \!\!\!\! 
  |\mu_{J,M}|^p.  
\end{align}
  The rest of the proof goes similarly as in the case $p\le 1$. Now we  should  choose $\varepsilon >0$ such that  $\varkappa-2\varepsilon- \sigma_p+s=\varkappa-2\varepsilon+s>0$ and $ \varkappa-2\varepsilon-s>0$. In other words, we need $\varkappa$ to satisfy $L>\varkappa > |s|$,  but this is again possible in view of \eqref{L-cond-wave}.
 
\emph{Step 4.} The proof of  the property (ii) is straightforward. Let $Q$ be some cube, $J\in\no$ and $M\in\zd$ such that $Q_{J,M}\subset Q$. Then we have 
\begin{align*}
|\mu_{J,M}| & \le  \sup_{k\in \zd} \, \Bigl(  \!\!\!\! { \mathop{\sum_{M\in\zd:}}_{Q_{J,M}\subset Q_{0,k}}} 
     \!\!\!\!  |\mu_{J,M}|^p\Bigr)^{1/p} { \sim }\, 
   2^{J\frac{\nd}{p}} \sup_{k\in\zd} \varphi(2^{-0}) 2^{(0-J)\frac{\nd}{p}} 
     \Big(  \!\!\!\!  { \mathop{\sum_{M\in\zd:}}_{Q_{J,M}\subset Q_{0,k}}} 
       \!\!\!\! |\mu_{J,M}|^p\Big)^{\frac1p} \\
& \leq  2^{J(\frac{\nd}{p}-s)} 2^{Js} \mathop{\sup_{\nu: \nu \leq J}}_{k\in \zd}\!  \varphi(2^{-\nu}) \,2^{(\nu-J)\frac{\nd}{p} }\Big(\!\! \mathop{\sum_{N\in\zd:}}_{Q_{J,N}\subset Q_{\nu,k}}\!\!|\mu_{J,N}|^p\Big)^{\frac1p} \\
& \le   c  2^{J(\frac{\nd}{p}- s)} \, \|\mu|n^s_{\varphi,p,\infty}\| \le  c  2^{J\varkappa} \, \|\mu|\n\|.
  \end{align*}
So the estimate holds with the same constant for any cube $Q$  and $\varkappa>(\frac{\nd}{p}- s)_+$. In view of \eqref{L-cond-wave} it is always possible to find $\varkappa$ such that $L>\varkappa>(\frac{\nd}{p}- s)_+$. This concludes the proof.
\end{proof}

\begin{remark}
 As in the paper \cite{HST} we do not claim the condition in \eqref{L-cond-wave}  to be sharp, the assumption on $L$ is just taken for convenience, following the argument in \cite{HST}. Moreover, for our purposes, that is, to transfer our sequence space results from Section~\ref{emb-seq} to the function space counterparts in Section~\ref{emb-func}, it is absolutely sufficient to find {\em some} number $L$ satisfying \eqref{L-cond-wave}. But we did not care for minimal assumptions. 
The result for the `classical' case $\varphi(t)=t^{\nd/u}$, $t>0$, $0<p\leq u<\infty$, can be found  in \cite[Thm. 4.5, Cor. 4.17]{MR-1} and \cite{Saw2}.
  \end{remark}


\section{Embeddings of generalised Besov-Morrey sequence spaces}\label{emb-seq}
First we deal with the embeddings of generalised Besov-Morrey sequence spaces $\n$, for the definitions we refer to Section~\ref{prelim}. These sequence spaces appear naturally when applying the wavelet decomposition result Theorem~\ref{waveletth} for generalised Besov-Morrey (function) spaces.

\begin{theorem}  \label{main}
Let $s_i\in\rr$, $0<p_i<\infty$, $0<q_i\leq \infty$, and $\varphi_i\in {\mathcal G}_{p_i}$, for $i=1,2$. 
We assume without loss of generality that $\varphi_1(1)=\varphi_2(1)=1$. 
Let $ \varrho=\min(1,\frac{p_1}{p_2})$ and $\alpha_j= \sup_{\nu\le j}\frac{\varphi_2(2^{-\nu})}{\varphi_1(2^{-\nu})^\varrho}$, $j\in \mathbb{N}_0$. 

There is a continuous embedding 
\begin{equation} \label{embed1}
\na \hookrightarrow \nb
\end{equation}
 if and only if 
 
\begin{align}\label{cond0}
\sup_{\nu\le 0}\frac{\varphi_2(2^{-\nu})}{\varphi_1(2^{-\nu})^\varrho} & < \infty , 
\intertext{and} 
\label{cond2}
\left\{   2^{j(s_2-s_1)} \alpha_j \frac{\varphi_1(2^{-j})^\varrho}{\varphi_1(2^{-j})}\right\}_j & \in \ell_{q^*}  \qquad \text{where}  \quad \frac{1}{q^*}=\left(\frac{1}{q_2}-\frac{1}{q_1}\right)_+  .
\end{align}
 The embedding  \eqref{embed1} is never compact.
\end{theorem}

\begin{proof}
{\em Step 1.}~ First we consider the sufficiency of the conditions \eqref{cond0}-\eqref{cond2}. Please note that it follows from  \eqref{cond0} that the supremum defining $\alpha_j$ is finite, so the sequence $(\alpha_j)_j$ is well defined. 
 
We start by proving some inequalities  for any fixed $j\in\no$.
If $p_2\le p_1$, i.e., $\varrho=1$,  then  we have the following inequality 
\begin{align}\label{thmm1:aa}
& \mathop{\sup_{\nu: \nu \leq j}}_{k\in \zd} \varphi_2(2^{-\nu}) 2^{(\nu-j)\frac{\nd}{p_2}}\biggl(\!\!\mathop{\sum_{m\in\zd:}}_{Q_{j,m}\subset Q_{\nu, k}}  
\!\! | \lambda_{j,m}|^{p_2} \biggr)^{\frac{1}{p_2}}  \\
& \qquad  \qquad \qquad  \qquad \leq  \alpha_j\, 
\mathop{\sup_{\nu: \nu \leq j}}_{k\in \zd} \varphi_1(2^{-\nu})  2^{(\nu-j)\frac{\nd}{p_1}}\biggl(\!\! \mathop{\sum_{m\in\zd:}}_{Q_{j,m}\subset Q_{\nu, k}}  \!\! 
| \lambda_{j,m}|^{p_1} \biggr)^{\frac{1}{p_1}} .  \nonumber
\end{align}
Indeed, for any $\nu\le j$ we have 
\begin{align}
&\varphi_2(2^{-\nu})   2^{(\nu-j)\frac{\nd}{p_2}} \biggl(\!\! \mathop{\sum_{m\in\zd:}}_{Q_{j,m}\subset Q_{\nu, k}}   \!\! | \lambda_{j,m}|^{p_2}  \biggr)^{\frac{1}{p_2}}  \notag \\
& \qquad \qquad  \leq \varphi_2(2^{-\nu})  2^{(\nu-j)\frac{\nd}{p_2}}  2^{(j-\nu)\nd\left( \frac{1}{p_2} -\frac{1}{p_1}\right)} \biggl(\!\! \mathop{\sum_{m\in\zd:}}_{Q_{j,m}\subset Q_{\nu, k}} 
\!\!   | \lambda_{j,m}|^{p_1}  \biggr)^{\frac{1}{p_1}}  \notag \\
&\qquad \qquad  \leq   \, \alpha_j \varphi_1(2^{-\nu})  2^{(\nu-j)\frac{\nd}{p_1}}  \biggl(\!\! \mathop{\sum_{m\in\zd:}}_{Q_{j,m}\subset Q_{\nu, k}}  \!\!  | \lambda_{j,m}|^{p_1}  \biggr)^{\frac{1}{p_1}},\notag
\end{align}
where the first inequality follows by H\"older's inequality. Taking the supremum over $\nu\le j$ and $k\in\zd$ we get \eqref{thmm1:aa}. 

If $p_1<p_2$, i.e., $\varrho<1$,  then 
\begin{align}\label{thmm1:a}
& \mathop{\sup_{\nu: \nu \leq j}}_{k\in \zd} \varphi_2(2^{-\nu}) 2^{(\nu-j)\frac{\nd}{p_2}}\biggl(\!\! \mathop{\sum_{m\in\zd:}}_{Q_{j,m}\subset Q_{\nu, k}}  \!\!
| \lambda_{j,m}|^{p_2} \biggr)^{\frac{1}{p_2}}  \\
& \qquad  \qquad \qquad  \qquad \leq  \alpha_j\, \frac{\varphi_1(2^{-j})^\varrho}{\varphi_1(2^{-j})}
\mathop{\sup_{\nu: \nu \leq j}}_{k\in \zd} \varphi_1(2^{-\nu})  2^{(\nu-j)\frac{\nd}{p_1}}\biggl(\!\!\mathop{\sum_{m\in\zd:}}_{Q_{j,m}\subset Q_{\nu, k}}  \!\!
| \lambda_{j,m}|^{p_1} \biggr)^{\frac{1}{p_1}} .  \nonumber
\end{align}
It is sufficient to prove  \eqref{thmm1:a} for sequences $(\lambda_{j,m})_m$  satisfying the following assumption
\begin{equation}   \label{thmm1:b}
 \mathop{\sup_{\nu: \nu \leq j}}_{k\in \zd} \varphi_1(2^{-\nu})  2^{(\nu-j)\frac{\nd}{p_1}}\biggl(\!\! \mathop{\sum_{m\in\zd:}}_{Q_{j,m}\subset Q_{\nu, k}}  \!\!
 | \lambda_{j,m}|^{p_1} \biggr)^{\frac{1}{p_1}}=1.
\end{equation}
In this case $  \varphi_1(2^{-j}) |\lambda_{j,m}|\leq 1$ for any $m$. So 
\begin{align}
 \varphi_1(2^{-j})^{p_2}\mathop{\sum_{m\in\zd:}}_{Q_{j,m}\subset Q_{\nu, k}} \!\!   | \lambda_{j,m}|^{p_2}  \leq 
 \varphi_1(2^{-j})^{p_1} \mathop{\sum_{m\in\zd:}}_{Q_{j,m}\subset Q_{\nu, k}}  \!\!  | \lambda_{j,m}|^{p_1}  \leq 2^{(j-\nu)\nd}   \left(\frac{\varphi_1(2^{-j})}{\varphi_1(2^{-\nu})}  \right)^{p_1}, \qquad \nu \leq j, \nonumber
\end{align}
and  
\begin{align}
\varphi_2(2^{-\nu})  2^{(\nu-j)\frac{\nd}{p_2}} \biggl(\!\!\mathop{\sum_{m\in\zd:}}_{Q_{j,m}\subset Q_{\nu, k}}  \!\!  | \lambda_{j,m}|^{p_2}  \biggr)^{\frac{1}{p_2}}  \leq 
\frac{\varphi_2(2^{-\nu})}{\varphi_1(2^{-\nu})^\varrho} \varphi_1(2^{-j})^{\varrho-1}  \le \alpha_j  \varphi_1(2^{-j})^{\varrho-1},  
\qquad \nu \leq j .\notag 
\end{align}
Taking the supremum we get \eqref{thmm1:a}. 

{\em Step 2}. ~Now we prove sufficiency. The inequality \eqref{thmm1:a} coincides with \eqref{thmm1:aa}  if  we take $\varrho=1$, so we can work with \eqref{thmm1:a} and $\varrho  \le 1$. 

From \eqref{thmm1:a}, for $j\in\no$ we have
\begin{align}
&2^{s_2j}    \mathop{\sup_{\nu: \nu \leq j}}_{k\in \zd} \varphi_2(2^{-\nu}) 2^{(\nu-j)\frac{\nd}{p_2}}\biggl(\!\!\mathop{\sum_{m\in\zd:}}_{Q_{j,m}\subset Q_{\nu, k}}  \!\!
 | \lambda_{j,m}|^{p_2} \biggr)^{\frac{1}{p_2}} \nonumber \\
& \qquad  \leq   2^{j(s_2-s_1)} \alpha_j \frac{\varphi_1(2^{-j})^\varrho }{\varphi_1(2^{-j}) }   2^{s_1j} 
 \mathop{\sup_{\nu: \nu \leq j}}_{k\in \zd} \varphi_1(2^{-\nu}) 2^{(\nu-j)\frac{\nd}{p_1}}\biggl(\!\! \mathop{\sum_{m\in\zd:}}_{Q_{j,m}\subset Q_{\nu, k}}  \!\!
 | \lambda_{j,m}|^{p_1} \biggr)^{\frac{1}{p_1}}.  \label{thmm1:c}
\end{align}

If $q_1=\infty$, thus $q^*=q_2$, by \eqref{thmm1:c} and  \eqref{cond2} we clearly get
$$
\|\lambda\mid n^{s_2}_{\varphi_2,p_2,q_2}\| \leq \Big\| \Bigl\{   2^{j(s_2-s_1)} \alpha_j \frac{\varphi_1(2^{-j})^\varrho }{\varphi_1(2^{-j}) }\Bigr\}_j \mid \ell_{q_2}\Big\|\;  \|\lambda \mid n^{s_1}_{\varphi_1,p_1,\infty}\|.
$$

 If $q_1<\infty$ and  $q_2 \geq q_1$,   {then $q^*=\infty$, and  \eqref{thmm1:c}  together with} \eqref{cond2} yield

$$
\|\lambda\mid n^{s_2}_{\varphi_2,p_2,q_2}\|\leq  \|\lambda\mid n^{s_2}_{\varphi_2,p_2,q_1}\| \leq \Big\| \Bigl\{    2^{j(s_2-s_1)} \alpha_j \frac{\varphi_1(2^{-j})^\varrho }{\varphi_1(2^{-j})} \Bigl\}_j \mid \ell_{\infty}\Big\|\;  \|\lambda \mid n^{s_1}_{\varphi_1,p_1,q_1}\|\, .
$$

Finally, in case of $q_2<q_1<\infty$,  by \eqref{thmm1:c}  and  H\"older's inequality we obtain
\begin{align}
\|\lambda\mid \nb\| &\leq  \Big\| \Bigl\{   2^{j(s_2-s_1)} \frac{\varphi_2(2^{-j})}{\varphi_1(2^{-j})}\Bigr\}_j \mid \ell_{\frac{q_1q_2}{q_1-q_2}}\Big\|\;  \|\lambda\mid \na\| \nonumber \\
& \leq  \Big\| \Bigl\{    2^{j(s_2-s_1)} \alpha_j \frac{\varphi_1(2^{-j})^\varrho }{\varphi_1(2^{-j}) }\Bigr\}_j \mid \ell_{q^*}\Big\|\;  \|\lambda\mid \na\|
\end{align}
thanks to $q^*= \frac{q_1q_2}{q_1-q_2}$, see \eqref{cond2}. 

{\em Step 3.} It remains to prove the necessity of the conditions. First we prove that the embedding \eqref{embed1} implies \eqref{cond0}.

{\em Substep 3.1}   We fix $j_0\ge 0$,  $\nu_0 \le j_0$ and
 consider the sequence  $\lambda^{(j_0,\nu_0)}$ defined as follows
 \begin{equation}\label{mnec1}
 \lambda^{(j_0,\nu_0)}_{j,m} = \begin{cases}
\varphi_1(2^{-\nu_0})^{-1} & \text{if}\qquad j=j_0\quad \text{and}\quad Q_{j,m}\subset Q_{\nu_0,0},\\
0 & \text{otherwise}. 
 \end{cases}
 \end{equation}   
 Then 
 \begin{align}\label{mnec2}
\varphi_1(2^{-\nu})2^{(\nu-j)\frac{\nd}{p_1}} \biggl(\!\! \mathop{\sum_{m\in\zd:}}_{Q_{j,m}\subset Q_{\nu_0,k}} \!\! |\lambda^{(j_0,\nu_0)}_{j,m}|^{p_1}\biggr)^{\frac{1}{p_1}} \le & \\ \frac{\varphi_1(2^{-\nu})}{\varphi_1(2^{-\nu_0})} &
\begin{cases}
1 & \text{if}\qquad j=j_0\quad \text{and}\quad Q_{\nu,k} \subset Q_{\nu_0,0},\\
2^{(\nu-\nu_0)\frac{\nd}{p_1}} & \text{if}\qquad j=j_0\quad \text{and}\quad Q_{\nu_0,0}\subset Q_{\nu,k},\\
0 & \text{otherwise}. 
\end{cases} \nonumber
 \end{align}
 The function $\varphi_1$ belongs to the class ${\mathcal G}_{p_1}$ therefore $\varphi_1(2^{-\nu})\varphi_1(2^{-\nu_0})^{-1}2^{(\nu-\nu_0)\frac{\nd}{p_1}} \le 1$ if $Q_{\nu_0,0}\subset Q_{\nu,k}$ and $\varphi_1(2^{-\nu})\varphi_1(2^{-\nu_0})^{-1}\le 1$ if $Q_{\nu,k} \subset Q_{\nu_0,0}$. 
 In consequence 
\begin{equation}\label{mnec3}
\|\lambda^{(j_0,\nu_0)} |\na\| = 2^{j_0s_1} \mathop{\sup_{\nu: \nu \leq j_0}}_{k\in \zd}\!  \varphi_1(2^{-\nu}) \,2^{(\nu-j_0)\frac{\nd}{p_1}}\biggl(\!\!\! \mathop{\sum_{m\in\zd:}}_{Q_{j_0,m}\subset Q_{\nu,k}}\!\!|\lambda_{j_0,m}^{(j_0,\nu_0)}|^{p_1}\biggr)^{\frac{1}{p_1}}= 2^{j_0s_1}.
\end{equation} 
In a similar way we prove that 
\begin{equation}\label{mnec4}
\|\lambda^{(j_0,\nu_0)}|\nb\| = 2^{j_0s_2} \mathop{\sup_{\nu: \nu \leq j_0}}_{k\in \zd}\!  \varphi_2(2^{-\nu}) \,2^{(\nu-j_0)\frac{\nd}{p_2}} \biggl(\!\!\! \mathop{\sum_{m\in\zd:}}_{Q_{j_0,m}\subset Q_{\nu,k}}\!\! |\lambda^{(j_0,\nu_0)}_{j_0,m}|^{p_2}\biggr)^{\frac{1}{p_2}}
= 2^{j_0s_2} \frac{\varphi_2(2^{-\nu_0})}{\varphi_1(2^{-\nu_0})}.
\end{equation}
So if the embedding \eqref{embed1} holds, then 
\begin{equation}\label{mnec4a}
\frac{\varphi_2(2^{-\nu_0})}{\varphi_1(2^{-\nu_0})} \le C 2^{j_0(s_1-s_2)}.
\end{equation}
Moreover  the constant $C$ is independent of $j_0$ and $\nu_0$. So if  $p_1\ge p_2$, i.e., $\varrho=1$,  we can fix  $j_0=0$. This  proves \eqref{cond0}.

{\em Substep 3.2.}   Let $p_1 < p_2$, i.e., $\varrho= \frac{p_1}{p_2}$. Once more we fix $j_0\in {\mathbb N}_0$ and $\nu_0\in \zz$ with $\nu_0\leq j_0$. Let $N\in \mathbb N$ be such that $1\le N\le 2^{(j_0-\nu_0)\nd}$. We define a sequence $\lambda^{(N)}=(\lambda^{(N)}_{j,m})$ such that 
\begin{align}\label{necm4b1}
(1) &\qquad  \lambda^{(N)}_{j,m}=1\quad\text{or}\quad\lambda^{(N)}_{j,m}=0\quad\text{for any}\quad (j,m),\\
(2) & \qquad  \lambda^{(N)}_{j,m}=0\quad\text{if}\quad j\not=j_0\quad\text{or}\quad Q_{j_0,m}\nsubseteq Q_{\nu_0,0},\\
(3) & \qquad  \lambda^{(N)}_{j_0,m}=1\quad\text{exactly}\quad N\quad\text{times},\\
(4) & \qquad \text{if} \quad Q_{\nu,k}\subset Q_{\nu_0,0}, \quad \nu_0<\nu<j_0 ,\quad  \text{then}\quad Q_{\nu,k}\quad\text{contains} \quad \text{at most}\quad  2^{\nd(\nu_0-\nu)}N +2  \nonumber\\
& \qquad \text{cubes}\quad Q_{j_0,m} \quad\text{such that}\quad \lambda^{(N)}_{j_0,m}=1 .\label{necm4b4}
\end{align}

If $N=2^{(j_0-\nu_0)\nd}$, then we can simply take $\lambda^{(N)}_{j_0,m} =1$ for any cube $Q_{j_0,m}\subset Q_{\nu_0,0}$. So let us assume  $N<2^{(j_0-\nu_0)\nd}$.   We put $\up{x} = \min\{k\in \zz:\; k\ge x \}  $, $x\in \rr$. 

Let $M_1=\up{2^{-\nd}N}$. If $M_1=1$, i.e., $N\le 2^\nd$, we put $\lambda^{(N)}_{j_0,m}=1$ for at most one cube $Q_{j_0,m}$ in any cube $Q_{\nu_0+1,k}\subset Q_{\nu_0,0}$ in such a way that we do not exceed the total number $N$ and we finish the construction.  \\
Let $M_1> 1$ and let $Q_{\nu_0,0}=\bigcup_{i=1}^{2^\nd} Q_{\nu_0+1,k_i}$. Now we represent $N$ as the following sum 
\begin{equation}
N=N^{(1)}_{k_1}+\ldots + N^{(1)}_{k_{2^\nd}}
\end{equation}     
where
\begin{equation}\label{mnec5}
N^{(1)}_{k_i} = 
\begin{cases} 
M_1 & \qquad \text{if} \qquad iM_1\le N\\
N-(i-1)M_1& \qquad \text{if} \qquad (i-1)M_1<N<iM_1,\\
0& \qquad \text{if} \qquad N = N^{(1)}_{k_1}+\ldots +N^{(1)}_{k_{i-1}}. 
\end{cases}
\end{equation}
We group the elements $\lambda^{(N)}_{j_0,m}$ of the sequence in such a way that exactly $N^{(1)}_{k_i}$ elements related to the cube $Q_{\nu_0+1,k_i}$ are equal to $1$. 

 Next we repeat the procedure for any cube $Q_{\nu_0+1,k_i}$. We define $M^{(2)}_{k_i}=\up{2^{-\nd}N^{(1)}_{k_i}}$. If $M^{(2)}_{k_i}=1$, i.e., $M^{(2)}_{k_i}\le 2^\nd$, we put $\lambda^{(N)}_{j_0,m}=1$ for at most one cube $Q_{j_0,m}$ in any cube $Q_{\nu_0+2,k}\subset Q_{\nu_0+1,k_i}$ in such a way that we do not exceed the total number $N^{(1)}_{k_i}$ and we finish the construction on the cube $Q_{\nu_0+1,k_i}$. \\
If $M^{(2)}_{k_i}> 1$ and  $Q_{\nu_0+1,k_i}=\bigcup_{j=1}^{2^\nd} Q_{\nu_0+2,k_j}$, then  we represent $N^{(1)}_{k_i}$ as a sum 
\begin{equation}
N^{(1)}_{k_i}=N^{(2)}_{k_i,k_1}+\ldots+N^{(2)}_{k_i,k_j}+\ldots + N^{(2)}_{k_i,k_{2^\nd}}
\end{equation}   
where   the numbers  $N^{(2)}_{k_i,k_j}$ are defined in a similar way to \eqref{mnec5} with $N^{(1)}_{k_i}$ instead of $N$ and 
$M^{(2)}_{k_i}$ instead of $M_1$. 

In the next steps we define $M^{(3)}_{k_i,k_j}=\up{2^{-\nd}N^{(2)}_{k_i,k_j}}$ and so on. The procedure  stops after at most $2^{(j_0-\nu_0)}$ steps. One can easily see that for any $\nu$, $\nu_0<\nu = \nu_0+\eta < j_0$ we have
\begin{equation}
N^{(\eta)}_{k_i,k_j,\ldots, k_\ell}\le 2^{-\nd(\nu-\nu_0)} N + \sum_{i=0}^{\eta-1} 2^{-\nd i}\le 2^{-\nd(\nu-\nu_0)} N + \frac{1}{1-2^{-\nd}}.
\end{equation}

Now we take  $j_0=0$ and $N=\up{2^{(j_0-\nu_0)\nd}\varphi_1(2^{-\nu_0})^{-p_1}}\le 2^{(j_0-\nu_0)\nd}$. If $Q_{\nu,k}\subset Q_{\nu_0,0}$, then 
\begin{align}\label{mnec6}
\varphi_1(2^{-\nu})2^{(\nu-j_0)\frac{\nd}{p_1}}  \biggl(\!\! \mathop{\sum_{m\in\zd:}}_{Q_{j_0,m}\subset Q_{\nu,k}}\!\! |\lambda^{(N)}_{j,m}|^{p_1}\biggr)^{\frac{1}{p_1}}
 \le 2^{\frac{1}{p_1}} \varphi_1(2^{-\nu})2^{(\nu-j_0)\frac{\nd}{p_1}}  \max( 2^{-\nd(\nu-\nu_0)} N, 2)^{\frac{1}{p_1}}  \\ 
\le 2^{\frac{1}{p_1}} \max\left(\frac{\varphi_1(2^{-\nu})}{\varphi_1(2^{-\nu_0})}, 2^{\frac{1}{p_1}}\varphi_1(2^{-\nu})2^{(\nu-j_0)\frac{\nd}{p_1}}\right) \le C
 \nonumber
 \end{align}
since $\nu_0<\nu\le 0$, $\varphi_1\in \mathcal{G}_{p_1}$ and  $\varphi_1(1)=1$. The constant $C$ is independent of $\nu_0$. 
In consequence $\|\lambda^{(N)}|\na\| \le C$. So if the embedding \eqref{embed1} holds, then 
\begin{align}
\frac{\varphi_2(2^{-\nu_0})}{\varphi_1(2^{-\nu_0})^\varrho} \le \varphi_2(2^{-\nu_0})2^{\nu_0\frac{\nd}{p_2}} \big(\varphi_1(2^{-\nu_0})^{-p_1}2^{-\nu_0\nd}\big)^\frac{1}{p_2} \le  \varphi_2(2^{-\nu_0})2^{\nu_0\frac{\nd}{p_2}} N^\frac{1}{p_2}\le \|\lambda^{(N)}|\nb\| \le C, \nonumber
\end{align}
proving  \eqref{cond0} when $p_1<p_2$.

{\em Step 4.} We prove that the assumptions  \eqref{cond2} is  necessary. 
If there exists $\nu_0\le 0$ such that $\alpha_0=\frac{\varphi_2(2^{-\nu_0)}}{\varphi_1(2^{-\nu_0})^\varrho}$, then also for  $j\in \mathbb{N}$ we can find $\nu_j\le 0$ such that $\alpha_j=\frac{\varphi_2(2^{-\nu_j)}}{\varphi_1(2^{-\nu_j})^\varrho}$. If the supremum defining $\alpha_0$ is not attained, then there exist   $\nu_0\le 0$, and  in consequence $\nu_j\le j$, such that
\begin{equation} \label{nu-j}
\frac{\varphi_2(2^{-\nu_j)}}{\varphi_1(2^{-\nu_j})^\varrho}\le  \alpha_j < 2\frac{\varphi_2(2^{-\nu_j)}}{\varphi_1(2^{-\nu_j})
^\varrho}, \quad j\in \mathbb{N}_0.
\end{equation}
We used the modified version of the sequences constructed in Substep 3.2. 

{\em Substep 4.1}
First we assume that $p_1\ge p_2$, i.e., $\varrho=1$. 
Let $q_1\leq q_2$, i.e., $q^*=\infty$, and $i\in \mathbb{N}_0$. We consider the sequence  $\lambda^{(i)}=(\lambda^{(i)}_{j,m})$  defined by 
\begin{equation} 
\lambda^{(i)}_{j,m}=\begin{cases}
2^{-is_1} \alpha_i\varphi_2(2^{-\nu_i})^{-1}  &\text{ if} \qquad j=i \quad \text{and}\quad Q_{i,m} \subset Q_{\nu_i,0}, \\
0 & \text{otherwise}. 
\end{cases}
\end{equation}  
We prove that there is a positive $C>0$ such that  $\|\lambda^{(i)}|\na\|\le C $ for any $i$. 

Let $\nu_i\le \nu\le i$. Then 
\begin{align}\label{mnec7}
\varphi_1(2^{-\nu})  2^{(\nu-i)\frac{\nd}{p_1}}\biggl(\!\! \mathop{\sum_{m\in\zd:}}_{Q_{i,m}\subset Q_{\nu, k}}  \!\!
| \lambda_{i,m}^{(i)}|^{p_1} \biggr)^{\frac{1}{p_1}} \le 
C 2^{-is_1}\frac{\varphi_1(2^{-\nu})}{\varphi_1(2^{-\nu_i})}   2^{(\nu-i)\frac{\nd}{p_1}}  2^{(i-\nu)\frac{\nd}{p_1}}\le C 2^{-is_1} .
\end{align}

If $\nu<\nu_i$, then 
\begin{align}\label{mnec8}
\varphi_1(2^{-\nu}) & 2^{(\nu-i)\frac{\nd}{p_1}}\biggl(\!\!\mathop{\sum_{m\in\zd:}}_{Q_{i,m}\subset Q_{\nu, k}}  \!\!
| \lambda_{i,m}^{(i)}|^{p_1} \biggr)^{\frac{1}{p_1}} \le   
\varphi_1(2^{-\nu_i}) \frac{\varphi_1(2^{-\nu})}{\varphi_1(2^{-\nu_i})} 2^{(\nu-i)\frac{\nd}{p_1}}\biggl(\!\!\mathop{\sum_{m\in\zd:}}_{Q_{i,m}\subset Q_{\nu_i, 0}}  \!\!
| \lambda_{i,m}^{(i)}|^{p_1} \biggr)^{\frac{1}{p_1}}  \nonumber \\
&\le \varphi_1(2^{-\nu_i})  2^{(\nu_i-i)\frac{\nd}{p_1}}\biggl(\!\! \mathop{\sum_{m\in\zd:}}_{Q_{i,m}\subset Q_{\nu_i, 0}}  \!\!
| \lambda_{i,m}^{(i)}|^{p_1} \biggr)^{\frac{1}{p_1}} \le C  2^{-is_1}
\end{align}
where the last but one inequality follows from the inclusion $\varphi_1\in \mathcal{G}_{p_1}$. 
The inequalities \eqref{mnec7} and \eqref{mnec8} give us $\|\lambda^{(i)}|\na\|\le C$. So if the embedding \eqref{embed1} holds, then 
\begin{align}
2^{i(s_2-s_1)}\alpha_i =  2^{is_2}\varphi_2(2^{-\nu_i})  2^{(\nu_i-i)\frac{\nd}{p_2}}\biggl(\!\!\mathop{\sum_{m\in\zd:}}_{Q_{i,m}\subset Q_{\nu_i, 0}} \!\!  | \lambda_{i,m}^{(i)}|^{p_2} \biggr)^{\frac{1}{p_2}} \le C \|\lambda^{(i)}|\na \|\le C.
\end{align}
Thus $\|\{2^{i(s_2-s_1)}\alpha_i\}_i |\ell_\infty\|\le C$ which is \eqref{cond2} in this case.

Now let $q_2<q_1$, i.e., $q^*<\infty$. Let $\mu=(\mu_j)_j\in \ell_{q_1}$ and $\|\mu|\ell_{q_1}\|=1$.  We consider the  sequence  $\lambda=(\lambda_{j,m})$ defined by the formula
\begin{equation}  \nonumber
\lambda_{j,m}=\begin{cases}
2^{-js_1} \alpha_j\varphi_2(2^{-\nu_j})^{-1} \mu_j &\text{ if} \qquad Q_{j,m}\subset Q_{\nu_j,0},\\
0 & \text{otherwise}. 
\end{cases}
\end{equation}
In the same way as above we show that  $\|\lambda|\na\|\le C \|\mu|\ell_{q_1}\|$. So if the embedding \eqref{embed1} holds, then 
\begin{align}
\sum_{j=0}^\infty 2^{j(s_2-s_1)q_2} \alpha_j^{q_2}|\mu_j
|^{q_2}  = &\sum_{j=0}^\infty 2^{js_2q_2} \varphi_2(2^{-\nu_j})^{q_2} 2^{q_2(\nu_j-j)\frac{\nd}{p_2}} \biggl(\!\!\mathop{\sum_{m\in\zd:}}_{Q_{j,m}\subset Q_{\nu_j, 0}} \!\!
 | \lambda_{j,m}|^{p_2} \biggr)^{\frac{q_2}{p_2}} \nonumber\\  
    &\le   \|\lambda|\nb\|^{q_2} \le   C   \|\lambda|\na\|^{q_2} \le C . \nonumber 
\end{align} 

Any element of the   sequence $\left(2^{j(s_2-s_1)} \alpha_j\right)_j$ is positive therefore
\begin{align*}
  \|2^{j(s_2-s_1)} \alpha_j|\ell_{q^*}\|^{q_2} &=
\|(2^{j(s_2-s_1)} \alpha_j)^{q_2}|\ell_{q^*/q_2}\|
\\
 &=\sup_{\|\varkappa|\ell_{(q^*/q_2)'}\|=1;\varkappa_j\ge 0 }\
\sum_{j=0}^\infty 2^{j(s_2-s_1)q_2} \alpha_j^{q_2}\varkappa_j\\
&\le 
\sup_{\|\mu|\ell_{q_1}\|=1}\
\sum_{j=0}^\infty 2^{j(s_2-s_1)q_2} \alpha_{j}^{q_2}|\mu_j|^{q_2} \le C , 
\nonumber
\end{align*}
since $(\frac{q^*}{q_2})'=\frac{q_1}{q_2}$.

{\em Substep 4.2} 
We deal now with the case $p_1 < p_2$, i.e., $\varrho= \frac{p_1}{p_2}$. Let $q_1\le q_2$, i.e., $q^*=\infty$, and $i\in\mathbb{N}$. 
Consider the construction  explained in Substep 3.2, with $j_0$ and $\nu_0$ replaced by $i$ and $\nu_i$, respectively, where  $\nu_i$ satisfies \eqref{nu-j}. Moreover, let 
$N_i=\up{2^{(i-\nu_i)\nd}\varphi_1(2^{-\nu_i})^{-p_1}  \varphi_1(2^{-i})^{p_1} }$ and let  $\lambda^{(N_i)}$ be the  sequence  described in the above mentioned Substep 3.2.
Define the sequence $\beta^{(i)}$ by
\begin{equation} 
\beta^{(i)}_{j,m}=\begin{cases}
2^{-is_1} \alpha_i\varphi_2(2^{-\nu_i})^{-1}  \varphi_1(2^{-\nu_i})^{\rho} \varphi_1(2^{-i})^{-1}   \lambda_{j,m}^{(N_i)}      &\text{ if} \qquad j=i \quad \text{and}\quad Q_{i,m} \subset Q_{\nu_i,0}, \\
0 & \text{otherwise}. 
\end{cases}
\end{equation}  
We prove that there is a positive $C>0$ such that  $\|\beta^{(i)}|\na\|\le C $ for any $i$. 

By using \eqref{nu-j} and the fact that $\varphi_1\in \mathcal{G}_{p_1}$, we have,  in case of $\nu_i\le \nu\le i$, that 
\begin{align} 
\varphi_1(2^{-\nu})  2^{(\nu-i)\frac{\nd}{p_1}}\biggl(\!\!\mathop{\sum_{m\in\zd:}}_{Q_{i,m}\subset Q_{\nu, k}}  \!\!
| \beta_{i,m}^{(i)}|^{p_1} \biggr)^{\frac{1}{p_1}} \le 
2 \frac{\varphi_1(2^{-\nu})}{\varphi_1(2^{-i})}   2^{-is_1}2^{(\nu-i)\frac{\nd}{p_1}} \left(  2^{-\nd(\nu-\nu_i) }N_i+\frac{1}{1-2^{-\nd}}  \right)^{\frac{1}{p_1}} \le C 2^{-is_1}, \nonumber
\end{align}
and, in case of $\nu<\nu_i$,  
\begin{align}
\varphi_1(2^{-\nu}) & 2^{(\nu-i)\frac{\nd}{p_1}}\biggl(\!\!\mathop{\sum_{m\in\zd:}}_{Q_{i,m}\subset Q_{\nu, k}}  \!\!
| \beta_{i,m}^{(i)}|^{p_1} \biggr)^{\frac{1}{p_1}} \le   
2 \frac{\varphi_1(2^{-\nu})}{\varphi_1(2^{-i})}   2^{-is_1}2^{(\nu-i)\frac{\nd}{p_1}} N_i^{\frac{1}{p_1}}    
 \le C 2^{-is_1} . \nonumber
\end{align}
The above  inequalities show that  $\|\beta^{(i)}|\na\|\le C $. So, if the embedding  \eqref{embed1} holds, then 
\begin{align}
C \geq  \|\beta^{(i)}|\nb \| & \geq  2^{s_2i} \varphi_2(2^{-\nu_i})2^{(\nu_i-i)\frac{\nd}{p_2}}\biggl(\mathop{\sum_{m\in\zd:}}_{Q_{i,m}\subset Q_{\nu_i, 0}}  | \beta_{i,m}^{(i)}|^{p_2} \biggr)^{\frac{1}{p_2}}  \geq \nonumber \\
& \geq 2^{(s_2-s_1)i} 2^{(\nu_i-i)\frac{\nd}{p_2}}  \alpha_i N_i^{\frac{1}{p_2}} \geq  2^{(s_2-s_1)i}  \alpha_i  \frac{\varphi_1(2^{-i})^\varrho}{\varphi_1(2^{-i})}. \nonumber
\end{align}
Thus $\big\|\bigl\{   2^{(s_2-s_1)i} \alpha_i \frac{\varphi_1(2^{-i})^\varrho}{\varphi_1(2^{-i})}\bigr\}_i |\ell_\infty\big\|\le C$.

Now let $q_2<q_1$, i.e., $q^*<\infty$. Let $\mu=(\mu_j)_j\in \ell_{q_1}$ and consider the  sequence  $\beta=(\beta_{j,m})$ defined by 
\begin{equation}  \nonumber
\beta_{j,m}=\begin{cases}
2^{-js_1} \alpha_j\varphi_2(2^{-\nu_j})^{-1}  \varphi_1(2^{-\nu_j})^{\rho} \varphi_1(2^{-j})^{-1}   \lambda_{j,m}^{(N_j)} \,   \mu_j &\text{ if} \qquad Q_{j,m}\subset Q_{\nu_j,0}\\
0 & \text{otherwise}. 
\end{cases}
\end{equation}
In the same way as above we show that  $\|\beta|\na\|\le C \|\mu|\ell_{q_1}\|$. So if the embedding \eqref{embed1} holds, then 
\begin{align}
\sum_{j=0}^\infty 2^{j(s_2-s_1)q_2}  \alpha_j^{q_2}\frac{\varphi_1(2^{-j})^{\varrho q_2}}{\varphi_1(2^{-j})^{q_2}} |\mu_j |^{q_2}  &= \sum_{j=0}^\infty 2^{js_2q_2} \varphi_2(2^{-\nu_j})^{q_2} 2^{q_2(\nu_j-j)\frac{\nd}{p_2}} \biggl(\!\!\mathop{\sum_{m\in\zd:}}_{Q_{j,m}\subset Q_{\nu_j, 0}} \!\!
 | \beta_{j,m}|^{p_2} \biggr)^{\frac{q_2}{p_2}} \nonumber\\  
    &\le   \|\beta|\nb\|^{q_2} 
    \le C . \nonumber 
\end{align} 
Thus 
$$\left\{ 2^{j(s_2-s_1)q_2} \alpha_j^{q_2}\frac{\varphi_1(2^{-j})^{\varrho q_2}}{\varphi_1(2^{-j})^{q_2}} {\mu_j}^{q_2} \right\}_j \in \ell_1 \qquad \text{for all } (\mu_j)_j\in \ell_{q_1}$$
which is equivalent to 
$$\left\{ 2^{j(s_2-s_1)q_2} \alpha_j^{q_2}\frac{\varphi_1(2^{-j})^{\varrho q_2}}{\varphi_1(2^{-j})^{q_2}} {\eta_j} \right\}_j \in \ell_1 \quad \text{for all } (\eta_j)_j\in \ell_{r_1}$$
with $r_1=\frac{q_1}{q_2}$.
But this implies that 
$$\left\{ 2^{j(s_2-s_1)q_2} \alpha_j^{q_2}\frac{\varphi_1(2^{-j})^{\varrho q_2}}{\varphi_1(2^{-j})^{q_2}} \right\}_j \in \ell_{r_1'}
$$
which means that 
$$\left\{ 2^{j(s_2-s_1)} \alpha_j \frac{\varphi_1(2^{-j})^{\varrho}}{\varphi_1(2^{-j})} \right\}_j \in \ell_{q^*}.$$

{\em Step 5.}\quad  It remains to prove the non-compactness of \eqref{embed1}.  This follows by the same method as in 
\cite[Thm.~3.2]{hs12}, i.e., we can take a sequence $\lambda^{(\mu)}=\{\lambda^{(\mu)}_{j,m}\}_{j,m}$, $\mu\in \mathbb{N}$, 
\[ \lambda^{(\mu)}_{j,m}=
\begin{cases}
1 ,& \quad \text{if}\quad j=0 \;\text{and}\quad m=(\mu,0,0,\ldots,0),\\
0, &\quad \text{otherwise}.
\end{cases}
\]
Then $\|\lambda^{(\mu)}|\na\|=1$ and  $\|\lambda^{(\mu_1)}-\lambda^{(\mu_2)}|\nb\|\ge 1$ if $\mu_1\ne\mu_2$.
\end{proof}

\begin{remark}\label{rem-n-tn}
Following the above proof one observes, that $\na\hookrightarrow\nb$ if and only if $\tna\hookrightarrow\tnb$, where the latter spaces have been introduced in Section~\ref{wavelet}. The first term in \eqref{norm-tn} can just be treated as the term with $j=0$ in the argument above.
\end{remark}

\begin{examples} \label{rmk2}
  We explicate  Theorem~\ref{main} for a few settings as mentioned in Example~\ref{exm-Gp}. 
\begin{enumerate}[{\bfseries (a)}]  
\item
In the particular case of $\varphi_i(t)=t^{\frac{\nd}{u_i}}$, $0<p_i\leq u_i<\infty$,  $i=1,2$, condition \eqref{cond0} means
$\frac{1}{u_2}\le \frac{1}{u_1}\min(1,\frac{p_1}{p_2})$ that is equivalent to 
\begin{equation} \label{cond-classic1}
u_1\leq u_2 \qquad \text{and} \qquad \frac{p_2}{u_2}\leq \frac{p_1}{u_1}.
\end{equation}
Moreover, since
$$
\left\{   2^{j(s_2-s_1)} \frac{\varphi_2(2^{-j})}{\varphi_1(2^{-j})}\right\}_j = \big\{   2^{j(s_2-s_1+\frac{\nd}{u_1}-\frac{\nd}{u_2})}\big\}_j\ ,
$$ 
we recover exactly the conditions for classical  Besov-Morrey sequence spaces in Theorem 3.2 of \cite{hs12}.
\item 
Besides the `classical' example given above, we consider the  functions  $\varphi_{u_i,v_i}$ defined by \eqref{example1}, $0<u_i,v_i<\infty$, $i=1,2$. Now one can easily calculate that condition \eqref{cond0}  is equivalent to $\frac{v_1}{v_2}\le \varrho$ and the condition \eqref{cond2} is equivalent to 
  \begin{equation}\label{loc-loc}
   \begin{cases}
    \frac{s_1-s_2}{\nd} > \max\left\{0,  \frac{1}{u_1}-\frac{1}{u_2}, \frac{p_1}{u_1}\left(\frac{1}{p_1}-\frac{1}{p_2}\right)\right\}, & \text{if}\quad q_1>q_2, \smallskip \\
    \frac{s_1-s_2}{\nd} \geq  \max\left\{0,  \frac{1}{u_1}-\frac{1}{u_2}, \frac{p_1}{u_1}\left(\frac{1}{p_1}-\frac{1}{p_2}\right)\right\}, & \text{if}\quad q_1\leq q_2. \end{cases}
  \end{equation}
Please note that  \eqref{loc-loc} coincides with the conditions formulated in  \cite{hs13} for embeddings of Besov-Morrey spaces defined on bounded domains. 
\item
Finally we return to the setting in Example~\ref{exm-Gp}(iii), 
\[
  \varphi_i(t)=\begin{cases} t^{\frac{\nd}{u_i}}, & 0<t<1, \\ 1, & t\geq 1,\end{cases}  
 \]
where $u_i\geq p_i$, $i=0,1$. Formally this can be seen as an extension of the previous example to $v_i=\infty$, $i=1,2$. Please note that the sequence spaces correspond via Theorem \ref{waveletth} to the local Besov-Morrey spaces, cf. Remark \ref{rem-coinc}.  Since $\sup_{t>1} \varphi_2(t)/\varphi_1(t)^\varrho =1$, {\eqref{cond0} is satisfied, thus} it remains to deal with the condition  {\eqref{cond2}},  which leads to \eqref{loc-loc} again.
\end{enumerate}
\end{examples}
 
 Next we collect a number of interesting and useful implications of Theorem~\ref{main}.

\begin{corollary}  \label{mainb}
Let $s_i\in\rr$, $0<p_i<\infty$, $0<q_i\leq \infty$, and $\varphi_i\in {\mathcal G}_{p_i}$, for $i=1,2$. 
Let $\varphi_1(1)=\varphi_2(1)=1$ and $ \varrho=\min(1,\frac{p_1}{p_2})$. 
We assume that the sequence $\alpha_j= \sup_{\nu\le j}\frac{\varphi_2(2^{-\nu})}{\varphi_1(2^{-\nu})^\varrho}$ converges to some $\alpha\geq 1$ and that \eqref{cond0} is satisfied.
\bit
\item[{\upshape\bfseries (i)}]
If $p_1\ge p_2$, then the embedding  \eqref{embed1}
is continuous  if and only if $s_1>s_2$ or $s_1=s_2$ and $q_1\le q_2$. 
\item[{\upshape\bfseries (ii)}]
If $p_1<p_2$, then the embedding  \eqref{embed1}
is continuous  if and only if
\begin{align}
\label{bcond2}
\left\{   2^{j(s_2-s_1)} \varphi_1(2^{-j})^{\frac{p_1}{p_2}-1}\right\}_j  \in \ell_{q^*}\quad \text{where}\quad \frac{1}{q^*} = \left(\frac{1}{q_2}-\frac{1}{q_1}\right)_+ .
\end{align}
\eit
\end{corollary}

We recall that $b^s_{p,q}$,  { $s\in \rr$, $0<p,q\le \infty$, denote the classical Besov sequence  spaces}, cf. Remark \ref{seq-besov-rem}. Now we extend the above definition to the case $p=\infty$.

We have the following observation from Lemma~\ref{equiv-norm}.

\begin{corollary}\label{cor-infty}
  Let $0<p<\infty$, $0<q\leq \infty$, $s\in \rr$, and $\varphi\in \Gp$.
  \bit
  \item[{\bfseries\upshape (i)}]
If $\ \inf\limits_{t>0} \varphi(t) >0$, then $\n \hookrightarrow b^s_{\infty,q}$.
  \item[{\bfseries\upshape (ii)}]
If $\ \sup\limits_{t>0} \varphi(t)<\infty$, then $b^s_{\infty,q} \hookrightarrow \n$.
\eit
In particular, if $\ 0<\inf\limits_{t>0} \varphi(t)\leq \sup\limits_{t>0} \varphi(t)<\infty$, then $\n = b^s_{\infty,q}$ (in the sense of equivalent norms).
\end{corollary}

\begin{remark}
This can be seen as some sequence space counterpart of Remark~\ref{rmk1}.
\end{remark}

  \begin{corollary}\label{diff_spaces}
    Let $s_i\in\rr$, $0<p_i<\infty$, $0<q_i\leq \infty$, and $\varphi_i\in {\mathcal G}_{p_i}$, for $i=1,2$.
Then 
\begin{equation} \label{nn-1}
\na = \nb \qquad \text{(in the sense of equivalent norms)}
\end{equation}
 if and only if 
 \begin{equation}\label{nn-2'}
s_1=s_2 \qquad \text{and}\qquad q_1=q_2, 
 \end{equation}
 and one of the two conditions
\begin{equation}\label{extremal}
  0<\inf\limits_{t>0}\ \varphi_i(t)\leq \sup\limits_{t>0}\ \varphi_i(t)<\infty,\quad i=1,2,
  \end{equation}
  or
 \begin{equation}\label{nn-2}
 p_1=p_2\qquad \text{and}\qquad \varphi_1(t)\sim \varphi_2(t),\quad t>0,
 \end{equation}
holds.  
  \end{corollary}

  \begin{proof}
    {\em Step 1}.~ Clearly \eqref{nn-2'} and \eqref{nn-2} imply \eqref{nn-1}, but also \eqref{nn-2'} and \eqref{extremal} lead to \eqref{nn-1} which can be seen as follows: either one checks directly the conditions \eqref{cond0} and \eqref{cond2}, or one uses Corollary~\ref{cor-infty} and observes that \eqref{extremal} leads to $\na=b^{s_1}_{\infty,q_1}$ and $\nb=b^{s_2}_{\infty,q_2}$. Hence \eqref{nn-2'} completes the proof of the sufficiency for \eqref{nn-1}.

    {\em Step 2}.~ Now we deal with the necessity. Here we apply Theorem~\ref{main} twice, that is, for $\na\hookrightarrow \nb$ and $\nb\hookrightarrow\na$. Thus \eqref{cond0} in both cases leads to
    \[
    \varphi_2(2^{-\nu})\leq c \varphi_1(2^{-\nu})^{\min(1,\frac{p_1}{p_2})} \leq c' \varphi_2(2^{-\nu})^{\min(\frac{p_1}{p_2},\frac{p_2}{p_1})},\quad \nu\leq 0,
    \]
    such that $\varphi_2(2^{-\nu})^{1-\min(\frac{p_1}{p_2},\frac{p_2}{p_1})} \leq c'' $, $\nu\leq 0$. This requires either $p_1=p_2$ and thus $\varphi_1(t)\sim \varphi_2(t)$, $t\geq 1$, or $\sup_{t>0}\varphi_i(t)<\infty$, $i=1,2$.
      
    If $p_1=p_2$ and $\varphi_1(t) \sim \varphi_2(t)$, $t\geq 1$, then
    $\alpha_j \sim \max_{\nu=0, \dots, j} \frac{\varphi_2(2^{-\nu})}{\varphi_1(2^{-\nu})}$, likewise  
    $\widetilde{\alpha}_j\sim \max_{\nu=0, \dots, j} \frac{\varphi_1(2^{-\nu})}{\varphi_2(2^{-\nu})}$. Thus \eqref{cond2} leads, in particular, to
    \[
  \max_{\nu=0, \dots, j} \frac{\varphi_2(2^{-\nu})}{\varphi_1(2^{-\nu})} \leq c\ 2^{j(s_1-s_2)} \leq  \min_{\nu=0, \dots, j} \frac{\varphi_2(2^{-\nu})}{\varphi_1(2^{-\nu})} 
    \]
for all $j\in\no$, i.e., $\varphi_1(t)\sim \varphi_2(t)$, $0<t\leq 1$, $s_1=s_2$, which finally implies $q_1=q_2$, as desired.
     
Assume now $\sup_{t>0}\varphi_i(t)<\infty$, where we may restrict ourselves to the case $p_1\neq p_2$. It is sufficient to show that there appears a contradiction if \eqref{extremal} is not satisfied, as then -- again in view of Corollary~\ref{cor-infty} -- \eqref{extremal} and \eqref{nn-1} yield $b^{s_1}_{\infty,q_1}=
b^{s_2}_{\infty,q_2}$ which is known to imply \eqref{nn-2'} finally.
So let us assume $p_1<p_2$, hence $\min(1,\frac{p_1}{p_2})=\frac{p_1}{p_2}$, $\min(1, \frac{p_2}{p_1})=1$.
Let $\varepsilon>0$, then there exists some $j_0=j_0(\varepsilon)\in\nat$ such that $\varphi_1(2^{-j})^{\frac{p_1}{p_2}-1} \geq (\inf_{t>0} \varphi_1(t)+\varepsilon)^{\frac{p_1}{p_2}-1} \geq c>0$ for $ j\geq j_0$. If \eqref{extremal} is not satisfied, then at least one of the sequences $(\alpha_j)_j$ or $(\widetilde{\alpha}_j)_j$ diverges,
\[\alpha_j = \sup_{\nu\leq j} \frac{\varphi_2(2^{-\nu})}{\varphi_1(2^{-\nu})^{\frac{p_1}{p_2}}} \xrightarrow[j\to\infty]{} \infty\qquad\text{or}\qquad 
\widetilde{\alpha}_j= \sup_{\nu\leq j} \frac{\varphi_1(2^{-\nu})}{\varphi_2(2^{-\nu})}\xrightarrow[j\to\infty]{} \infty.
\]
Let us assume that $\varphi_2(t)\rightarrow 0$ for $t\to 0$. Then $\alpha_j \geq 1$, $j\in\no$, and $\widetilde{\alpha}_j\rightarrow \infty$ for $j\to\infty$. Consequently \eqref{cond2} applied to $\na\hookrightarrow \nb$ leads to $s_1\geq s_2$, but the second embedding requires $s_2>s_1$. This is a contradiction.
  \end{proof}

Next we study the special situation when $\varphi_1=\varphi_2=\varphi$.

    \begin{corollary}\label{same_phi}
      Let $s_i\in\rr$, $0<p_i<\infty$, $0<q_i\leq \infty$ for $i=1,2$, $\varphi\in {\mathcal G}_{\max(p_1,p_2)}$, and $\frac{1}{q^*}=(\frac{1}{q_2}-\frac{1}{q_1})_+$.
      \bit
      \item[{\upshape\bfseries (i)}] Let $p_1\geq p_2$. Then 
\begin{equation} 
\nfa \hookrightarrow \nfb 
\end{equation}
 if and only if 
 \begin{equation}\label{nn-3}
   \begin{cases}
     s_1>s_2, & \text{if}\ q_1>q_2, \\ 
     s_1\geq s_2, & \text{if}\ q_1\leq q_2. \end{cases}
 \end{equation}
      \item[{\upshape\bfseries (ii)}]
Let $p_1<p_2$. Then 
\begin{equation} 
\nfa \hookrightarrow \nfb 
\end{equation}
 if and only if 
 \begin{equation}\label{nn-4}
\sup_{t>0}\ \varphi(t)<\infty\qquad \text{and}\quad 
\left\{2^{j(s_2-s_1)} \varphi(2^{-j})^{\frac{p_1}{p_2}-1}\right\}_j \in \ell_{q^*}\ .
 \end{equation}
\eit
    \end{corollary}

  \begin{proof}
    If $p_1\geq p_2$, then (using the notation of Theorem~\ref{main})  $\varrho=1$ and $\alpha_j\equiv 1$, and \eqref{cond0} is automatically satisfied. Moreover, \eqref{cond2} reduces to the question whether $\{2^{j(s_2-s_1)}\}_j\in\ell_{q^*}$, i.e., \eqref{nn-3}, which completes the proof of (i).

In case of $p_1<p_2$, $\varrho=\frac{p_1}{p_2}<1$ and \eqref{cond0} is obviously equivalent to the first condition in \eqref{nn-4}, that is, $  \sup_{t>0}\ \varphi(t)<\infty$. Moreover, in that case $\alpha_j \sim 1$, such that \eqref{cond2} can be rewritten as the second part of \eqref{nn-4}. 
  \end{proof}

  \begin{remark}
    Note that a sufficient condition for $\varphi$ in (ii) is -- in addition to $\sup_{t>0}\ \varphi(t)<\infty$ -- that
    \[
    \frac{s_1-s_2}{\nd}>\frac{p_1}{p_2}\left(\frac{1}{p_1}-\frac{1}{p_2}\right)
    \]
using the properties of $\varphi\in \mathcal{G}_{\max(p_1,p_2)}=\mathcal{G}_{p_2}$.
  \end{remark}

\begin{example}\label{ex_same_phi}
  We explicate Corollary~\ref{same_phi} for some function $\varphi$. We restrict ourselves to the situation $p_1<p_2$, i.e.,  Corollary~\ref{same_phi}(ii).
{ Let}
  \[
  \varphi(t)=\begin{cases} t^{\frac{\nd}{p_2}} \left(1+|\log t|\right)^a, & 0<t<1, \\ 1, & t\geq 1,\end{cases}
  \]
  where $a\in\rr$. Since $\sup_{t>0} \varphi(t)<\infty$, we deal with the second condition in \eqref{nn-4},  which leads to
 \[ \begin{cases}
    \frac{s_1-s_2}{\nd} > \frac{p_1}{p_2}\left(\frac{1}{p_1}-\frac{1}{p_2}\right), & \text{or}, \smallskip\\
    \frac{s_1-s_2}{\nd} = \frac{p_1}{p_2}\left(\frac{1}{p_1}-\frac{1}{p_2}\right)\quad\text{and}\quad a\geq 0, & \text{if}\quad q_1\leq q_2,\quad\text{or},
\smallskip\\
    \frac{s_1-s_2}{\nd} = \frac{p_1}{p_2}\left(\frac{1}{p_1}-\frac{1}{p_2}\right)\quad\text{and}\quad a> \frac{1}{q^*} \frac{p_2}{p_2-p_1}, & \text{if}\quad q_1> q_2. \end{cases}
  \]
We  {used } that in this case
$$ 
   \left\{2^{j(s_2-s_1)} \varphi(2^{-j})^{\frac{p_1}{p_2}-1}\right\}_j = 
   \left\{2^{j(s_2-s_1-\nd \frac{p_1}{p_2}(\frac{1}{p_1}-\frac{1}{p_2}))} (1+j)^{a(\frac{p_1}{p_2}-1) }\right\}_j \ .
$$ 
  \end{example}

  Now we focus on embeddings where either the target or the source space is a Besov sequence space, {for what we recall that} 
  \[
  b^s_{p,q} = n^s_{p,p,q}, \quad 0<p<\infty, \quad 0<q\leq \infty, \quad s\in\rr.
  \]

  \begin{corollary}\label{gen_morrey_besov}
    Let $s_i\in\rr$, $0<p_i<\infty$, $0<q_i\leq \infty$ for $i=1,2$, and $\varphi_1\in {\mathcal G}_{p_1}$. Denote again
     $\frac{1}{q^*}=(\frac{1}{q_2}-\frac{1}{q_1})_+$. Then 
        \begin{equation}\label{nn-5-0}
    \na \hookrightarrow b^{s_2}_{p_2,q_2}
   \end{equation}
        if and only if
        \begin{equation}\label{nn-5}
          p_1\leq p_2,\qquad \varphi_1 (t) \sim t^{\frac{\nd}{p_1}}, \ t\geq 1,\qquad \text{and}\quad\left\{2^{j(s_2-s_1) }\varphi_1(2^{-j})^{\frac{p_1}{p_2}-1}\right\}_j \in \ell_{q^*}.
   \end{equation}
  \end{corollary}
  
  \begin{proof}
    We apply Theorem~\ref{main} (and its notation) with $\varphi_2(t)=t^{\frac{\nd}{p_2}}$, $t>0$. Thus \eqref{cond0} is equivalent to
    \[\varphi_1(t)\geq \ c\ t^{\frac{\nd}{p_2}\frac{1}{\varrho}},\quad t\geq 1.\]
    On the other hand, $\varphi_1\in\mathcal{G}_{p_1}$ implies $\varphi_1(t)\leq t^{\frac{\nd}{p_1}}$, $t\geq 1$, hence this results in $p_1\leq p_2$ and $\varphi_1(t)\sim t^{\frac{\nd}{p_1}}$, $t\geq 1$, which is the first part of \eqref{nn-5}. We concentrate on \eqref{cond2} and observe that $\alpha_j$ is bounded, since
          \[1\leq \alpha_j=\sup_{\nu\leq j} \varphi_1(2^{-\nu})^{-\frac{p_1}{p_2}} 2^{-\nu\frac{\nd}{p_2}} \sim \max_{\nu=0, \dots, j} \varphi_1(2^{-\nu})^{-\frac{p_1}{p_2}} 2^{-\nu\frac{\nd}{p_2}} \leq \max_{\nu=0, \dots, j}  2^{\nu\frac{\nd}{p_2}} 2^{-\nu\frac{\nd}{p_2}} = 1,
\]
        where we used again $\varphi_1\in\mathcal{G}_{p_1}$, this time leading to $\varphi_1(2^{-\nu})\geq 2^{-\nu\frac{\nd}{p_1}}$, $\nu\in\no$. Thus \eqref{cond2} corresponds to the second part in \eqref{nn-5}.
  \end{proof}

\begin{example} 
{We illustrate Corollary \ref{gen_morrey_besov} for $\varphi_1$  given by
$$
\varphi_{1}(t)=
\begin{cases}
t^{\nd/u} & \text{if}\qquad t\le 1,\\ 
t^{\nd/p_1} & \text{if}\qquad t >1, 
\end{cases}
$$
with $p_1\leq u <\infty$,  a special case of  \eqref{example1}.  It turns out that in such a case \eqref{nn-5-0} holds if and only if $p_1\leq p_2$ and 
 \begin{equation*}
   \begin{cases}
     \frac{s_1-s_2}{\nd} >  \frac{p_1}{u}\left(\frac{1}{p_1}-\frac{1}{p_2}\right) 
     , & \text{if}\quad q_1>q_2, \quad \text{or},\smallskip \\
     \frac{s_1-s_2}{\nd} \geq  \frac{p_1}{u}\left(\frac{1}{p_1}-\frac{1}{p_2}\right)
     , & \text{if}\quad q_1\leq q_2. \end{cases}
  \end{equation*} 
When $u=p_1$ this is the sequence space  counterpart of  \cite[Cor.~3.7]{hs12}.} 
 \end{example}

    \begin{corollary}\label{besov_gen_morrey}
      Let $s_i\in\rr$, $0<p_i<\infty$, $0<q_i\leq \infty$ for $i=1,2$, and $\varphi_2\in {\mathcal G}_{p_2}$. Denote again $\frac{1}{q^*}=(\frac{1}{q_2}-\frac{1}{q_1})_+$.
      \bit
 \item[{\upshape\bfseries (i)}]
Let $p_1\leq p_2$. Then
    \[
b^{s_1}_{p_1,q_1}  \hookrightarrow  \nb
\]
if and only if  
        \begin{equation}\label{nn-6}
\{2^{j(s_2-s_1+\frac{\nd}{p_1})}\varphi_2(2^{-j})\}\in\ell_{q^*} . 
        \end{equation}
 \item[{\upshape\bfseries (ii)}]
Let $p_1>p_2$. Then 
    \[
b^{s_1}_{p_1,q_1}  \hookrightarrow  \nb
\]
if and only if 
\begin{align}\label{cond0'}
\sup_{t\geq 1} t^{-\frac{\nd}{p_1}} \varphi_2(t) & < \infty , 
\intertext{and} 
\label{cond2'}
\left\{   2^{j(s_2-s_1)} \sup_{0\leq \nu\leq j} 2^{\nu \frac{\nd}{p_1}} \varphi_2(2^{-\nu})
\right\}_j & \in \ell_{q^*}.
\end{align}
\eit
  \end{corollary}
  
  \begin{proof}
    Part (ii) exactly corresponds to Theorem~\ref{main} with $\varphi_1(t)=t^{\frac{\nd}{p_1}}$, $t>0$, and $\varrho=1$. As for part (i), now with $\varrho=\frac{p_1}{p_2}$, \eqref{cond0} reads as
    \[
    \sup_{\nu\leq 0} 2^{\nu\frac{\nd}{p_2}} \varphi_2(2^{-\nu})\leq c,
    \]
    but this is always true in view of $\varphi_2\in\mathcal{G}_{p_2}$. By the same argument, 
\[ \alpha_j \sim \max_{\nu=0, \dots, j} 2^{\nu\frac{\nd}{p_2}} \varphi_2(2^{-\nu}) = \varphi_2(2^{-j}) 2^{j\frac{\nd}{p_2}},
 \]   
 which leads to
 \[ 2^{j(s_2-s_1)} \alpha_j \varphi_1(2^{-j})^{\varrho-1} = 2^{j(s_2-s_1+\frac{\nd}{p_2})} \varphi_2(2^{-j}) 2^{-j\frac{\nd}{p_1}(\frac{p_1}{p_2}-1)} =  
 2^{j(s_2-s_1+ \frac{\nd}{p_1})} \varphi_2(2^{-j}) ,\]
 such that \eqref{cond2} coincides with \eqref{nn-6}.
  \end{proof}

  \begin{example}
We consider a model function for part (ii), i.e., when $p_1>p_2$. Recall that in this case there is no continuous embedding for classical Besov-(Morrey) spaces on $\rd$. Let 
    \[
\varphi_2(t) = \begin{cases} t^{\frac{\nd}{u_1}}, & t\geq 1, \\ 
t^{\frac{\nd}{u_2}}, & 0<t\leq 1, \end{cases}
  \]
  where $u_1\geq p_1$ and $u_2\geq p_2$. We may even admit $u_1=\infty$ with the understanding that $\varphi_2(t)=1$ for $t\geq 1$. Then $\varphi_2\in\mathcal{G}_{p_2}$, \eqref{cond0'} is satisfied, and \eqref{cond2'} leads to
  \[
\begin{cases}  \frac{s_1-s_2}{\nd} \geq \left(\frac{1}{p_1}-\frac{1}{u_2}\right)_+ & \text{if}\quad q_1\leq q_2, \smallskip \\ 
\frac{s_1-s_2}{\nd} > \left(\frac{1}{p_1}-\frac{1}{u_2}\right)_+ & \text{if}\quad q_1> q_2.\end{cases}
  \]
In particular, $u_1=u_2=p_1$ is admitted, such that $\varphi_1(t)=\varphi_2(t) =t^{\frac{\nd}{p_1}}$ then and we recover our result from Corollary~\ref{same_phi}(i) for this case.
 \end{example}

\section{Embeddings of generalised Besov-Morrey function spaces}\label{emb-func}

Now we deal with embeddings of the function spaces. We   benefit from our sequence space result Theorem~\ref{main} and the wavelet characterisation of the function spaces, {cf.} Theorem \ref{waveletth}. The following statement is the immediate  consequence of the just mentioned theorems, recall also Remark~\ref{rem-n-tn}.

\begin{theorem}\label{fs}  
Let $s_i\in\rr$, $0<p_i<\infty$, $0<q_i\leq \infty$, and $\varphi_i\in {\mathcal G}_{p_i}$, for $i=1,2$. 
We assume without loss of generality that $\varphi_1(1)=\varphi_2(1)=1$. 

There is a continuous embedding 
\begin{equation} \label{embed1fs}
{\mathcal N}^{s_1}_{\varphi_1,p_1,q_1}(\rd) \hookrightarrow {\mathcal N}^{s_2}_{\varphi_2,p_2,q_2}(\rd)
\end{equation}
if and only if
\eqref{cond0} and \eqref{cond2} are satisfied, using the notation of Theorem~\ref{main}.
 
The embedding  \eqref{embed1fs} is never compact.
\end{theorem}

\begin{remark}
  If $\varphi_i(t)= t^\frac{\nd}{u_i}$, $i=1,2$, then Theorem~\ref{fs} coincides with \cite[Theorem~3.3]{hs12}. 
\end{remark}

\begin{corollary}
Let $s\in\rr$, $0<p<\infty$, $0<q\leq \infty$, and  $\varphi\in {\mathcal G}_{r}$, $r=\frac{p^2}{p_1}$.  If $0<p_1\le p$, then
 \[ 
  {\mathcal N}^{s+\frac{\nd}{p}}_{p,p_1,q}(\rd)\hookrightarrow {\mathcal N}^{s}_{\varphi,p,q}(\rd). \]  
\end{corollary}

\begin{proof}
	The statement follows directly from Theorem \ref{fs} with $\varphi_1(t)= t^\frac{\nd}{p}$ and $\varphi_2(t)= \varphi(t)$. 
\end{proof}	

Now we collect further consequences of Theorem~\ref{fs} parallel to our approach in Section~\ref{emb-seq}.

\begin{corollary}\label{diff_spaces2}
    Let $s_i\in\rr$, $0<p_i<\infty$, $0<q_i\leq \infty$, and $\varphi_i\in {\mathcal G}_{p_i}$, for $i=1,2$.
Then 
\begin{equation} \label{nn-1fs}
\MBa(\rd) = \MBb(\rd) \qquad \text{(in the sense of equivalent norms)}
\end{equation}
 if and only if we have the equalities  \eqref{nn-2'}  
 and one of the two conditions \eqref{extremal} or \eqref{nn-2}
holds.  
  \end{corollary}

\begin{proof}
This is the function space version of Corollary~\ref{diff_spaces}.
\end{proof}

 \begin{corollary}\label{same_phi_fs}
      Let $s_i\in\rr$, $0<p_i<\infty$, $0<q_i\leq \infty$ for $i=1,2$, $\varphi\in {\mathcal G}_{\max(p_1,p_2)}$, and $\frac{1}{q^*}=(\frac{1}{q_2}-\frac{1}{q_1})_+$. Then 
\begin{equation} 
\MBfa(\rd) \hookrightarrow \MBfb(\rd) 
\end{equation}
 if and only if  $p_1\ge p_2$ and \eqref{nn-3} holds or  $p_1 < p_2$ and \eqref{nn-4} holds.
    \end{corollary} 

 \begin{proof}
This is the counterpart for function spaces of Corollary~\ref{same_phi}.
\end{proof}

  \begin{corollary}\label{gen_morrey_besov_fs}
    Let $s_i\in\rr$, $0<p_i<\infty$, $0<q_i\leq \infty$ for $i=1,2$, and $\varphi_1\in {\mathcal G}_{p_1}$. Denote again $\frac{1}{q^*}=(\frac{1}{q_2}-\frac{1}{q_1})_+$. Then 
    \[
    \MBa(\rd) \hookrightarrow B^{s_2}_{p_2,q_2}(\rd)
        \]
        if and only if the conditions \eqref{nn-5} hold.
  \end{corollary}

\begin{proof}
This corresponds to Corollary~\ref{gen_morrey_besov}.
\end{proof}
  
In combination with the well-known embedding $B^0_{r,1}(\rd)\hookrightarrow L_r(\rd)$, $1\leq r<\infty$, we thus obtain from Corollary~\ref{gen_morrey_besov_fs} the following result.

  \begin{corollary}\label{gen_morrey_Lr}
    Let $s\in\rr$, $0<p<\infty$, $0<q\leq \infty$, $\varphi\in {\mathcal G}_{p}$, with $\varphi (t) \sim t^{\frac{\nd}{p}}$, $ t\geq 1$. Assume $1\leq r<\infty$ with $r\geq p$, and let $\frac{1}{q'}=(1-\frac{1}{q})_+$. Then 
    \[
    \MB(\rd) \hookrightarrow L_r(\rd)
        \]
        if $\ \left\{2^{-js }\varphi(2^{-j})^{\frac{p}{r}-1}\right\}_j \in \ell_{q'}$.
  \end{corollary}

  Finally we return to the situation studied in Corollary~\ref{besov_gen_morrey}.
  
  \begin{corollary}\label{besov_gen_morrey_fs}
      Let $s_i\in\rr$, $0<p_i<\infty$, $0<q_i\leq \infty$ for $i=1,2$, and $\varphi_2\in {\mathcal G}_{p_2}$. Denote again $\frac{1}{q^*}=(\frac{1}{q_2}-\frac{1}{q_1})_+$.
  Then 
\[  B^{s_1}_{p_1,q_1}(\rd)  \hookrightarrow  \MBb(\rd)\]
 if and only if 
 \bit 
   \item[{\upshape\bfseries (i)}]
  $p_1\leq p_2$ and the condition  \eqref{nn-6} holds  
  \item[ ] or
  \item[{\upshape\bfseries (ii)}]
    $p_1> p_2$ and the conditions  \eqref{cond0'}--\eqref{cond2'} hold.
    \eit   
\end{corollary}

  \begin{remark}
 Obviously one can also explicate Theorem~\ref{fs} for the example functions, similar to Examples~\ref{rmk2}, Example~\ref{ex_same_phi} etc. 
For instance, we can prove that the formula \eqref{loc-loc} gives  sufficient and necessary conditions for the embedding  of two local Besov-Morrey spaces. 
\end{remark}
  
In the end we study some endpoint situations in Corollaries~\ref{gen_morrey_besov_fs}-\ref{besov_gen_morrey_fs}, recall also the sequence space counterpart in Corollary~\ref{cor-infty}. We begin with an extension of Corollary~\ref{gen_morrey_Lr} to $r=\infty$. Recall our notation $\frac{1}{q'}=(1-\frac1q)_+$ for $0<q\leq\infty$.

\begin{corollary}
Let $s\in\rr$, $0<p<\infty$, $0<q\leq \infty$, and $\varphi\in {\mathcal G}_{p}$. Assume that
\begin{equation}\label{suff-Linf}
\left\{ 2^{-js} \varphi(2^{-j})^{-1} \right\}_{j\in\no} \in \ell_{q'} \ .
\end{equation}
Then
\[
{\mathcal N}^{s}_{\varphi,p,q}(\rd) \hookrightarrow L_\infty(\rd).
\]
\end{corollary}

\begin{proof}
We apply Theorem~\ref{fs} with $\varphi_1=\varphi$, $\varphi_2\equiv 1$, $s_1=s$, $s_2=0$, $p_1=p_2=p$, $q_1=q$, $q_2=1$ and hence $q^*=q'$. Thus, in view of \eqref{M1p}, 
\[
{\mathcal N}^{s}_{\varphi,p,q}(\rd) \hookrightarrow B^0_{\infty,1}(\rd)\hookrightarrow L_\infty(\rd),
\]
where the latter embedding is well-known.
\end{proof}

\begin{remark}
  In case of $\varphi(t)=t^{\frac{\nd}{u}}$, $0<p\leq u<\infty$, \eqref{suff-Linf} reads as 
\[
\begin{cases}
s > \frac{\nd}{u}, & \text{if}\quad 1<q\leq\infty, \\ s\geq \frac{\nd}{u}, & \text{if}\quad 0<q\leq 1, 
\end{cases}
\]
and this is even known to be also necessary for the embedding 
${\mathcal N}^{s}_{u,p,q}(\rd) \hookrightarrow L_\infty(\rd)$, cf. \cite[Prop.~5.5]{hs13}. 
\end{remark}

Next we return to Sawano's observation  that for $\ \inf_{t>0} \varphi(t)>0$, then $\mathcal{M}_{\varphi,p}(\rd)\hookrightarrow L_\infty(\rd)$, while 
$\ \sup_{t>0} \varphi(t)<\infty\ $ implies $ L_\infty(\rd)\hookrightarrow\mathcal{M}_{\varphi,p}(\rd)$, leading as a special case to \eqref{M1p}, cf. \cite{Saw18}. For convenience, let us denote 
these conditions by
\begin{align*}
\mathbf{(I)}  &\qquad \inf_{t>0} \varphi(t)>0,\qquad \text{and}\\
\mathbf{(S)} &\qquad \sup_{t>0} \varphi(t)<\infty.
  \end{align*}

\begin{corollary}
Let $s\in\rr$, $0<p<\infty$, $0<q\leq \infty$, and $\varphi\in {\mathcal G}_{p}$. 
\begin{itemize}
\item[{\bfseries\upshape (i)}]
 If  $\varphi$ satisfies $\mathbf{(I)}$, then 
 \[ {\mathcal N}^{s}_{\varphi,p,q}(\rd) \hookrightarrow B^s_{\infty,q}(\rd).\] 
\item[{\bfseries\upshape (ii)}]
If  $\varphi$ satisfies $\mathbf{(S)}$,  then 
\[B^s_{\infty,q}(\rd) \hookrightarrow {\mathcal N}^{s}_{\varphi,p,q}(\rd)  .\] 
\end{itemize}
Hence if $\varphi$ satisfies $\mathbf{(I)}$ and $\mathbf{(S)}$, i.e., $\varphi\sim 1$, then
\[ {\mathcal N}^{s}_{\varphi,p,q}(\rd) = B^s_{\infty,q}(\rd)
\] 
(in the sense of equivalent norms).
\end{corollary}
\begin{proof}
	The first statement  is also a consequence of Theorem \ref{fs} with $\varphi_1(t)= \varphi(t)$ and $\varphi_2(t) = \varphi_0(t)\equiv 1$, since $\mathcal{M}_{\varphi_0, p}(\rd) = L_{\infty}(\rd)$, recall \eqref{M1p}.  Moreover if $\sup_{t>0} \varphi(t)<\infty$, then $L_{\infty}(\rd)\hookrightarrow  \mathcal{M}_{\varphi, p}(\rd)$. This implies the second embedding. 
\end{proof}	

We conclude our paper with a closer look on the consequences of $\mathbf{(I)}$ and $\mathbf{(S)}$ for the standard embedding 
\begin{equation} \label{embed1fs'}
{\mathcal N}^{s_1}_{\varphi_1,p_1,q_1}(\rd) \hookrightarrow {\mathcal N}^{s_2}_{\varphi_2,p_2,q_2}(\rd).
\end{equation}

\begin{corollary}
Let $s_i\in\rr$, $0<p_i<\infty$, $0<q_i\leq \infty$, and $\varphi_i\in {\mathcal G}_{p_i}$, for $i=1,2$. 
We assume without loss of generality that $\varphi_1(1)=\varphi_2(1)=1$.
\bit
\item[{\bfseries\upshape (i)}]
  Assume that $\varphi_1$ satisfies $\mathbf{(I)}$. Then  \eqref{embed1fs'}  {holds} if and only if \eqref{cond0} is satisfied and $\ \{2^{j(s_2-s_1)}\}_j \in \ell_{q^*}$.
\item[{\bfseries\upshape (ii)}]
Assume that $\varphi_2$ satisfies $\mathbf{(I)}$ while $\varphi_1$ does not. Then  \eqref{embed1fs'}  {holds} if and only if \eqref{cond0} is satisfied and $\ \{2^{j(s_2-s_1)} \varphi_1(2^{-j})^{-1}\}_j \in \ell_{q^*}$.
\item[{\bfseries\upshape (iii)}]
Assume that $\varphi_2$ satisfies $\mathbf{(S)}$ while $\varphi_1$ does not satisfy $\mathbf{(I)}$. Then  \eqref{embed1fs'}  {holds} if and only if \eqref{cond2} holds. In particular   \eqref{embed1fs'} {holds} if  $\ \{2^{j(s_2-s_1)} \varphi_1(2^{-j})^{-1}\}_j \in \ell_{q^*}$.
\item[{\bfseries\upshape (iv)}]
  Assume that $\varphi_1$ satisfies $\mathbf{(S)}$. Then  \eqref{embed1fs'} {holds} if and only if also $\varphi_2$ satisfies $\mathbf{(S)}$ and 
  \eqref{cond2} holds.
    \eit
\end{corollary}

\begin{proof}
  We begin with (i). In view of Theorem~\ref{fs} it remains to verify that $\mathbf{(I)}$ for $\varphi_1$ together with $\ \{2^{j(s_2-s_1)  }\}_j \in \ell_{q^*}$ is equivalent to \eqref{cond2}. However, since $\varphi_2$ is nondecreasing and $\varphi_1$ satisfies $\mathbf{(I)}$, we get that $1\leq \alpha_j \leq c\ \alpha_0$ and $\varphi_1(2^{-j})^{\varrho-1} \leq c'$ such that \eqref{cond2} follows.

  Next we deal with (ii). This time we need to show that the assumptions on $\varphi_2$ and $\ \{2^{j(s_2-s_1)} \varphi_1(2^{-j})^{-1}\}_j \in \ell_{q^*}$  is equivalent to \eqref{cond2}. But using the boundedness of $\varphi_2$ and the monotonicity of $\varphi_1$ we obtain $c \varphi_1(2^{-j})^{-\varrho}\le \alpha_j \leq \varphi_1(2^{-j})^{-\varrho}$ for sufficiently large $j$ since $\varphi_1$ does not satisfy $\mathbf{(I)}$. But this together with our assumption leads to \eqref{cond2}.

  First observe that the assumed boundedness of $\varphi_2$ from above in (iii) together with the boundedness of $\varphi_1$ from below already imply \eqref{cond0}. The boundedness of $\alpha_j \leq \varphi_1(2^{-j})^{-\varrho}$ follows in the same way as in (ii).
  
It remains to deal with (iv). If $\varphi_2$ satisfies $\mathbf{(S)}$, then \eqref{cond0} is a consequence of $\varphi_1(2^{-\nu})\geq \varphi_1(1)=1$ for $\nu\leq 0$. The rest follows by Theorem~\ref{fs}.
\end{proof}

\bibliographystyle{alpha}	
\def\cprime{$'$}

\bigskip

\noindent Dorothee D. Haroske\\
\noindent  Institute of Mathematics,
Friedrich Schiller University Jena, 07737 Jena, Germany\\
\noindent {\it E-mail}:  \texttt{dorothee.haroske@uni-jena.de}

\bigskip

\noindent Susana D. Moura\\
\noindent  University of Coimbra, CMUC, Department of Mathematics,  EC Santa Cruz, 3001-501 Coimbra, Portugal\\
\noindent {\it E-mail}:  \texttt{smpsd@mat.uc.pt}

\bigskip

\noindent Leszek Skrzypczak\\
\noindent  Faculty of Mathematics and Computer Science,
Adam Mickiewicz University, Ul. Uniwersytetu Pozna\'nskiego 4, 61-614 Pozna\'n,
Poland\\
\noindent {\it E-mail}:  \texttt{lskrzyp@amu.edu.pl}

\end{document}